\documentclass[12pt,letterpaper]{amsart}
\usepackage{latexsym}
\usepackage{bigints}

\usepackage[dvips]{graphics}

\usepackage{float}
\usepackage{hyperref}
\usepackage{amssymb}
\usepackage{amsthm}
\usepackage{amsmath}
\usepackage{amstext}
\usepackage{amscd}

\usepackage{graphicx}

\usepackage{color}

\usepackage{tikz-cd}

\usepackage{ifmtarg}
\makeatletter
  \newcommand{\TODO}[1]{\@ifmtarg{#1}{\emph{\textcolor{red}{\textbf{TODO}}}~}{\textcolor{red}{ \emph{\textbf{TODO:}~#1~}}}}
  \newcommand{\FORUS}[1]{\@ifmtarg{#1}{\emph{\textcolor{blue}
  {\textbf{TODO}}}~}{\textcolor{blue}{ \emph{\textbf{FOR US:}~#1~}}}}
\makeatother

\numberwithin{equation}{section}

\theoremstyle{plain}%default
\newtheorem{thm}{Theorem}[section]
\newtheorem{lem}[thm]{Lemma}

\newtheorem{theorem*}{Theorem}[]

\newtheorem{prop}[thm]{Proposition}

\theoremstyle{definition}
\newtheorem{defi}[thm]{Definition}

\newtheorem{rmk}[thm]{Remark}

\theoremstyle{definition}

\newtheorem{example}[thm]{Example}

\theoremstyle{remark}

\newtheorem{question}{Question}

\newcommand{\R}{\mathbb{R}}

\newcommand{\bigslant}[2]{{\raisebox{.2em}{$#1$}\left/\raisebox{-.2em}{$#2$}\right.}}

\usepackage{enumitem}
\setlist[itemize]{labelindent=.6em, itemindent=1em, leftmargin=!, label=\textbullet}
\usepackage{ifmtarg}

\title{Sobolev sheaves on the plane}

\author{M'hammed OUDRANE}
\address {IRMAR (UMR 6625), Université de Rennes, Campus de Beaulieu,
35042 Rennes Cedex, France.}
\email{m-hammed.oudrane@univ-rennes.fr}


\begin{document}

\begin{abstract}
In this paper, we show that for any integer $k \in \mathbb{N}$ there exists a Sobolev sheaf (in the sense of Lebeau) on any definable site of $\mathbb{R}^2$ that agrees with Sobolev spaces on cuspidal domains. We also provide a complete computation of the cohomology of these sheaves using the notion of good direction introduced by Valette. 
\end{abstract}
\maketitle

\vspace{1cm}
\section{Introduction}

Sheaves of functional spaces on the subanalytic topology (introduced by Kashiwara and Schapira in \cite{KS}) are important objects in algebraic analysis, which involves studying solutions of  $\mathcal{D}$-modules as a generalization of linear partial differential equations. The most famous example is the sheaf of tempered distributions on the subanalytic site of a complex manifold, introduced by Kashiwara \cite{KA} to provide an elegant solution to the Riemann-Hilbert problem. In this paper, our focus is on sheaves composed of Sobolev functions. For $s\in \mathbb{R}$, the presheaf of $\mathbb{C}$-vector spaces
$$U \subset \mathbb{R}^n \mapsto W^{s,2}(U)=\{F_{\mid U} \; : \; F \in W^{s,2} (\mathbb{R}^n)  \},$$
is not always a sheaf (as shown by Lebeau \cite{L}). This is related to the fact that if $U\subset \mathbb{R}^n$ is an open subanalytic set with a non-Lipschitz boundary $\partial U$, then the space $W^{s,2} (U)$ doesn't exhibit favorable properties. More precisely, it is well known that in this case, Sobolev functions on $U$ are not necessarily restrictions of Sobolev functions on $\mathbb{R}^n$, and this gives rise to various issues. The aim of this paper is to find for $s>0$ an optimal sheafification of Sobolev spaces $W^{s,2}$ on the definable site (of a fixed o-minimal structure). Optimal in the sense that for $U\subset \mathbb{R}^n$, the space $W^{s,2} (U)$ will be modified only if it is necessary. 

In \cite{L}, Lebeau proved that for any $s<0$, there exists an object $\mathcal{F}^s$ in the derived category of sheaves on the subanalytic topology of $\mathbb{R}^n$, such that for any open bounded subanalytic set $U\subset \mathbb{R}^n$ with Lipschitz boundary, the complex $\mathcal{F}^s (U)$ is concentrated in degree $0$ and equal to the classical Sobolev space $W^{s,2}(U)$. The proof relies on the linear subanalytic site introduced by Guillermou and Schapira in \cite{GS}.

For \(k \in \mathbb{N}\), we construct a sheaf \(\mathcal{F}^k\) of distributions on the definable site of \(\mathbb{R}^2\) such that, for any small open set \(U \subset \mathbb{R}^2\) (either an open \(L\)-regular cell or an open set that is locally \(L\)-regular near its boundary), we have
\begin{center}
 $\mathcal{F}^k (U)=W^{k,2}(U)$.
 \end{center} 
 In a more formal way, our main result in this paper will be:\\
 
 \textbf{Main result:} Let  $\mathcal{A}$ be an o-minimal structure  on the real field $(\mathbb{R} , + , \cdot)$. Then, for any $k\in \mathbb {N}$, there exists a sheaf $\mathcal{F}^k$ on the definable site (associated to $\mathcal{A}$) of $\mathbb {R}^2$ such that, for any $U\subset \mathbb {R} ^2$ open definable bounded L-regular cell, we have
 $\mathcal{F}^k (U)=W^{k,2}(U)$.
 Moreover, for any $U\subset \mathbb {R}^2$ open definable bounded and for any $j>1$, we have
\begin{center}
$H^j (U, \mathcal{F}^k)=0$. 
\end{center}

Additionally, if $U$ has no punctured disk singularities, then
\begin{center}
$H^j (U, \mathcal{F}^k)=\left\{
                        \begin{array}{ll}
                  \mathcal{F}^k (U) \; \; \; \; \; \;  \; \; \; \; if \;  j=0 \\
                        \{0\}  \; \; \; \; \; \; \; \; \; \; \; \; \; if \; j\geqslant 1. \\
                        \end{array}
                        \right.$
\end{center}

This sheaf is unique (thanks to L-regular decomposition (see \cite{Pa3})) and agrees with $W^{k,2}$ on domains with Lipschitz boundaries. The idea of the  construction is based on understanding the local obstructions for $W^{k,2}$ to be a sheaf. Note that again thanks to L-regular decomposition, for $s\in ]-\frac{1}{2} , \frac{1}{2}[$ the presheaf $U\mapsto W^{s,2}(U)$ is a sheaf (see Lebeau \cite{L}). The obstructions are present for  $s>0$ big enough to have embedding of $W^{s,2}$ into at least the space of continuous functions. In the two dimensional case, the construction is explicit because the Lipschitz structure of definable open subsets in $\mathbb{R}^2$ has an explicit classification. Computation
of the cohomology is less obvious and requires more technical work. \\

The paper is organized as follows:
\begin{itemize}
    \item \textbf{Section 2:} We recall the basic concepts of o-minimal structures that are necessary for the context of this paper.
    \item \textbf{Section 3:} We present the definitions of Sobolev spaces $W^{s,2}$ as introduced in \cite{L}, along with the classical Stein extension theorem (Theorem ~\ref{thm3.2}).
    \item \textbf{Section 4:} We provide the definitions of definable sites and sheaves on definable sites (after Kashiwara and Schapira \cite{KS}), followed by the discussion of the sheafification problem for Sobolev spaces.
    \item \textbf{Section 5:} We discuss the spaces $W^{s,2}$ for $s\in ]-\frac{1}{2} , \frac{1}{2}[$.  
    \item \textbf{Section 6:} Here, we define the presheaf $\mathcal{F}^k$ (for $k \in \mathbb{N}$) of Hilbert spaces on a fixed definable site of  $\mathbb{R}^2$ and subsequently prove its sheaf property.
    \item \textbf{Section 7:} This is a core section focusing on a complete cohomology computation, establishing $\mathcal{F}^k$ as a Sobolev sheaf.
    \item \textbf{Section 8:} We give a sufficient condition to extend our method to Sobolev spaces $W^{s,2}$ for $s\in \mathbb{R}$. Notably, this offers a categorical proof of Lebeau's result from \cite{L}, affirming the validity of the Mayer-Vietoris sequence on domains with Lipschitz boundaries.
    \item \textbf{Section 9:} Finally, we provide remarks and insights concerning challenges in higher dimensions and the case of Sobolev spaces with fractional degrees of differentiability $s \in \R$.
\end{itemize}
\vspace{0.2cm}

\textbf{Acknowledgment.} The author is very grateful to Adam Parusi\'nski and Armin Rainer for their help and support, and the long hours of discussion they 
devoted to the author during the preparation of this work. The author extends warm and profound thanks to Georges Comte and  Guillaume Valette for reading this manuscript, and for the valuable comments, remarks, and suggestions. Part of the work has been done at University of Vienna, where the author’s research was funded by the Austrian Science Fund (FWF) Project P 32905-N. I am very grateful for the kind hospitality and the excellent working conditions. The author sincerely thanks the referees for their valuable comments, helpful corrections, and for providing Example 6.4.

\section{Definitions and Preliminaries}

\subsection{Notations:}
\begin{itemize}

\item $\mathcal{P}(X)$ is the set of subsets of $X$.
\item For a finite set $X$,  $\sharp X$ is the number of elements of $X$.
\item $B(v,r)$ represents the open ball with radius $r$ and center $v$, and $\overline{B}(v,r)$ represents the closed ball with radius $r$ and center $v$. Alternatively, notations $B_r(v)$ and $\overline{B}_r(v)$ might be used. 
\item $C(v,r)$ represents the sphere with radius $r$ and center $v$, i.e.,
$$C(v,r)=\overline{B}_r(v) \setminus B_r(v)=\{x\in \mathbb{R}^n \; : \; d(x,v)=r  \}.$$
\item For a definable set $X\subset \mathbb{R}^n$, $X^{reg}$ is the set of points $x\in X$ where $X$ is a $C^{1}$ manifold nearby $x$.
\item For $v\in \mathbb{R}^{n-1}$, $\pi_{v}:\mathbb{R}^{n}\longrightarrow \mathbb{R}^{n-1}$ is the linear projection parallel to $Vect((v,1))$.
\item For a set $A\subset \mathbb{R}^{n}\times \mathbb{R}^{m}$ and $x_{0}\in \mathbb{R}^{n}$, we denote by $A_{x_{0}}$ the set
\begin{center}
 $A_{x_{0}}=\{y\in \mathbb{R}^{m} \; : \; (x_{0},y)\in A \}$.
 \end{center} 
\item $\overline{A}$ refers to the topological closure of $A$.
\item For a set $U \subset \mathbb{R}^{n}$, $\partial U$ represents the boundary of $U$, i.e., $\partial U=\overline{U}\setminus U$. 
\item For an open set $U \subset \mathbb{R}^2$, a point $x \in \mathbb{R}^2$, and $r > 0$, we denote by $U_r(x)$ the open set $U \cap B(x, r)$.
\item $\mathbb{N}$ denotes the set of nonnegative integers.
\item For a map $f:A \to B$, $\Gamma_f$ denotes the graph of $f$.
\item For two functions $f:A\to [0,+\infty [$ and $g:A\to [0,+\infty [$, we write $f \lesssim g$ if there is $C>0$ such that $f(x) \leqslant C g(x)$ for all $x\in A$. 
\item For two functions $f:A\to \mathbb{R}$ and $g:A\to \mathbb{R}$ with $f<g$, $\Gamma (A,f,g)$ (or simply $\Gamma (f,g)$) denotes the set:
$$\Gamma (A,f,g) = \{(x,y)\in A\times \mathbb{R} \; : \; f(x)<y<g(x)  \}.$$
\item If $u,v\in \mathbb{R}^2 \setminus \{ 0 \}$, $\angle (u,v)$ represents the angle between $u$ and $v$ with respect to the anticlockwise orientation.
\item For $U\subset \mathbb{R}^n$ open, $\mathcal{D}(U)$ represents the topological vector space of $C^\infty$ functions with compact support in $U$, and $\mathcal{D}'(U)$ represents the space of continuous linear forms on $\mathcal{D}(U)$. 
\item $H^j(X,\mathcal{F})$ denotes the $j$-th cohomology group of the sheaf $\mathcal{F}$ on the topological space $X$.
\item If $\mathcal{A}$ is an o-minimal structure on the real field $(\mathbb{R}, +, \cdot)$, then $X_{\mathcal{A}}(\mathbb{R}^n)$ represents the site on $\mathbb{R}^n$ where open sets are open bounded definable (in $\mathcal{A}$) subsets of $\mathbb{R}^n$, and coverings are finite. $D^+(X_{\mathcal{A}}(\mathbb{R}^n))$ denotes the derived category of bounded below complexes of sheaves on the site $X_{\mathcal{A}}(\mathbb{R}^n)$. If $\mathcal{A}$ is the structure of globally subanalytic sets, then $X_{sa}(\mathbb{R}^n)$ is used instead of $X_{\mathcal{A}}(\mathbb{R}^n)$.

\end{itemize}

\subsection{O-minimal structures}
An o-minimal structure on the field $(\mathbb{R},+,\cdot)$ is a sequence $\mathcal{A}=(\mathcal{A}_{n})_{n\in \mathbb{N} }$ such that for any $n$, we have:
\begin{itemize}

\item   $\mathcal{A}_{n} $ is a Boolean subalgebra of $ \mathcal{P}(\mathbb{R}^{n})$.
\item   $\mathcal{A}_{n} $ contains all the real algebraic subsets of $\mathbb {R}^n$.
\item   $\pi (\mathcal{A} _{n})\subset \mathcal{A}_{n-1}$, where $\pi : \mathbb{R}^{n} \longrightarrow \mathbb{R}^{n-1}$ is the standard projection.
\item For all  $(n,m)\in \mathbb {N} ^2$: $\mathcal{A}_{n} \times \mathcal{A}_{m} \subset \mathcal{A}_{n+m}$.
\item   For any $A\in \mathcal{A}_{1} $, $A$ is a finite union of points and intervals.
\end{itemize}
For a fixed o-minimal structure $\mathcal{A}$:
\begin{itemize}
\item Elements of $\mathcal{A}_{n}$ are called definable sets.
\item If $A\in \mathcal{A}_{n}$ and $B\in \mathcal{A}_{m}$, then a map $f:A\longrightarrow B$ is called a definable map if its graph is a definable set.

\end{itemize}
We refer to \cite{VD} for the fundamentals of o-minimal geometry.

\textbf{Cell decomposition:}

For a given positive integer $p$, a definable set $C$ in $\mathbb{R}^{n}$ is referred to as a $C^{p}$-cell if:

\begin{itemize}

\item[case $n=1$:] $C$ is either a point or an open interval.
\item[case $n\geq 2$:] $C$ is one of the following:

\begin{itemize}
\item[$\bullet$] $C=\Gamma_{\phi}$ (the graph of $\phi$), where $\phi:B\longrightarrow \mathbb{R}$ is a $C^{p}$ definable function, and $B$ is a $C^{p}$-cell in $\mathbb{R}^{n-1}$. 
\item[$\bullet$] $C=\Gamma(\phi , \varphi)=\{ (x,y) \in B\times \mathbb{R} \; : \; \phi(x)<y<\varphi(x)  \}$, where $\phi$ and $\varphi$ are two $C^{p}$ definable functions on a $C^{p}$-cell $B$, satisfying $\phi < \varphi$ with the possibility of $\phi=-\infty$ or $\varphi =+\infty$.

\end{itemize}

\end{itemize}

A $C^{p}$-cell decomposition of $\mathbb{R}^{n}$ is defined by induction as follows:

\begin{itemize}

\item[$\bullet$] A $C^{p}$-cell decomposition of $\mathbb{R}$ is a finite partition consisting of points and open intervals.
\item[$\bullet$] A $C^{p}$-cell decomposition of $\mathbb{R}^{n}$ is a finite partition $\mathcal{P}$ of $\mathbb{R}^{n}$ by $C^{p}$-cells. It is required that $\pi(\mathcal{P})$ is a $C^{p}$-cell decomposition of $\mathbb{R}^{n-1}$, where $\pi: \mathbb{R}^{n}\longrightarrow \mathbb{R}^{n-1}$ is the standard projection, and $\pi(\mathcal{P})$ is the family

\begin{center}
$\pi(\mathcal{P})=\{\pi(A) \; : \; A\in \mathcal{P} \}$.
\end{center}

\end{itemize}

\begin{thm}
Let $p\in \mathbb{N}$ and $\{ X_{1},...,X_{n} \}$ be a finite family of definable sets of $\mathbb{R}^{n}$. Then there is a $C^{p}$-cell decomposition of $\mathbb{R}^{n}$ compatible with this family, i.e. each $X_{i}$ is a union of some cells.
\end{thm}

\begin{proof}
See \cite{M} or \cite{VD}.
\end{proof} 
Now we can define the dimension of a definable set.
Take $X$ a definable subset of $\mathbb{R}^{n}$ and $\mathcal{C}$ a cell decomposition of $\mathbb{R}^{n}$ compatible with $X$, then we define the dimension 

\begin{center}
$dim_{\mathcal{C}}(X)=max\{ dim(C) \; : \; C\subset X \; and \; C\in \mathcal{C} \}$.
\end{center}
This number does not depend on $\mathcal{C}$, we denote it by $dim(X)$.\\

\vspace{0.1cm}

Throughout the text, we assume $\mathcal{A}$ is an o-minimal structure on $(\mathbb{R},+,.)$.
\vspace{0.1cm}
\subsection{L-regular decomposition}
L-regular cells (Lipschitz cells) were introduced by A. Parusi\'nski to establish the existence of Lipschitz stratification for subanalytic sets (\cite{Pa3}, see also \cite{K}).
\begin{defi}
Let $X\subset \mathbb{R}^n$ be a definable subset. We say that $X$ is L-regular if:
\begin{itemize}
\item[$\bullet$] $X$ is a point if $\dim(X)=0$.
\item[$\bullet$] $X$ is an open interval if $\dim(X)=1$ and $n=1$.
\item[$\bullet$] If $\dim(X)=n$ (with $n>1$), then there exists $X'\subset \mathbb{R}^{n-1}$ that is L-regular, along with two $C^1$ definable functions with bounded derivatives $\phi _1 , \phi _2 :X' \longrightarrow \mathbb{R}$ where $\phi _1 <\phi _2 $, satisfying  
$$X=\{ (x',x_n)\in X' \times \mathbb{R} \; : \; \phi _1 (x')< x_n < \phi _2 (x')  \} .$$
\item[$\bullet$] If $\dim(X)=k<n$, then $X$ is the graph of a $C^1$ definable map
$ \phi : X' \longrightarrow \mathbb{R}^{n-k}$ with bounded derivatives on $Int(X')$, where $X' \subset \mathbb{R}^k$
 is L-regular and of dimension $k$.
 
\end{itemize}

\end{defi}
We will also say that $A$ is L-regular if it becomes so after a linear change of coordinates.

\begin{figure}[H]
  \centering
  \includegraphics[scale=0.4]{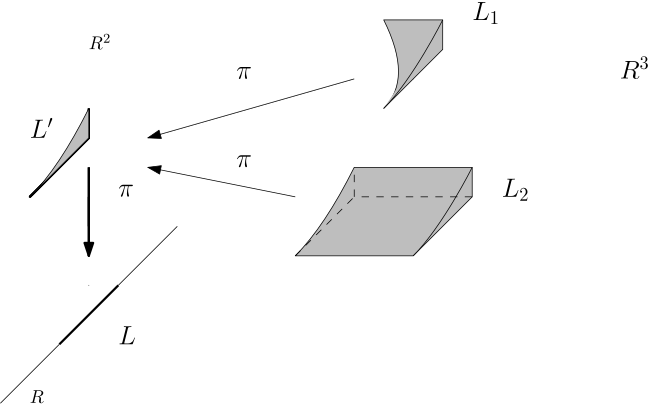}
  \caption{Example of building L-regular cells by induction.}   
\end{figure}

\begin{thm}\label{L-reg-decomposition}
Let $X_1, \ldots, X_l$ be definable subsets of $\mathbb{R}^n$. Then, there exists a finite definable partition $(L_k)_k$ of $\bigcup _i X_i$ that is compatible with each $X_i$, and each element $L_k$ is $L$-regular.

\begin{proof}
See \cite{Pa3} or \cite{K}.
\end{proof}

\end{thm}

\section{$L^2$-based Sobolov spaces revisited.}

For proofs and further details of the statements in this section, we refer to \cite{L} or \cite{J}. Let $n\in \mathbb{N}$. We denote:

\begin{itemize}
\item $\mathcal{S}(\mathbb{R}^n)$ as the space of Schwartz functions ($C^\infty$-functions that vanish at infinity along with all their derivatives, decaying faster than any polynomial).
\item $\mathcal{S}'(\mathbb{R}^n)$ as the topological dual of $\mathcal{S}(\mathbb{R}^n)$.
\end{itemize}
And we have natural continuous injections
$$\mathcal{S}(\mathbb{R}^n) \subset L^2(\mathbb{R}^n) \subset \mathcal{S}'(\mathbb{R}^n).$$

We recall the Fourier transform
\begin{center}
$u\in \mathcal{S}(\mathbb{R}^n) \mapsto \widehat{u} \in \mathcal{S}(\mathbb{R}^n)$,
\end{center}
where 
\begin{equation}
\widehat{u}(y)=\frac{1}{(2\pi)^{\frac{n}{2}}}\int_{\mathbb{R}^n} e^{-iy\cdot x}u(x)dx.
\end{equation}

By duality, the Fourier transform extends in a canonical way to $\mathcal{S}'(\mathbb{R}^n)$. Finally, for $s\in \mathbb{R}$, we recall the Sobolev space 
$$W^{s,2}(\mathbb{R}^n) =\{ u \in \mathcal{S}'(\mathbb{R}^n) \; : \;  \| u \| _{W^{s,2}(\mathbb{R}^n)}=\sqrt{\int _{\mathbb{R}^n} (1+\left| y \right|^2 )^s \left| \widehat{u}(y) \right|^2  dy } <+\infty    \},$$
with the natural dense inclusions (for $s\geqslant 0$)

\begin{center}
$\mathcal{D}(\mathbb{R}^n) \subset \mathcal{S}(\mathbb{R}^n) \subset L^2(\mathbb{R}^n) \subset W^{s,2}(\mathbb{R}^n) \subset \mathcal{S}'(\mathbb{R}^n)$.
\end{center}

An equivalent way to define $W^{s,2} (\mathbb{R}^n)$ is as follows:

\begin{itemize}
\item[$\bullet$] For $k\in \mathbb{N}$
$$W^{k,2} (\mathbb{R}^n)=\{f\in L^2 (\mathbb{R}^n) \; : \; \forall |\alpha| \leqslant k , \; \partial^\alpha f \in L^2 (\mathbb{R}^n)      \},$$
where $\partial^\alpha f$ denotes the distributional derivative of $f$ for $\alpha \in \mathbb{N}^n$.

\item[$\bullet$] For $s\in ]k,k+1[$ for some $k\in \mathbb{N}$, then $W^{s,2}$ is the interpolation space
$$W^{s,2} (\mathbb{R}^n)=[W^{k+1,2} (\mathbb{R}^n),W^{k,2} (\mathbb{R}^n)]_{s-k}.$$

\item[$\bullet$] For $s<0$, $W^{s,2} (\mathbb{R}^n)$ is the topological dual
\begin{center}
$W^{s,2} (\mathbb{R}^n)= (W^{-s,2} (\mathbb{R}^n))'$.
\end{center}
\end{itemize}

Let \(F \subset \mathbb{R}^n\) be a closed set. We define $W^{s,2}_F(\mathbb{R}^n)$ to be the closed subspace of \(W^{s,2}(\mathbb{R}^n)\) consisting of all
distributions whose support is contained in \(F\), equipped with the induced norm.

Take $s\geqslant0$ and $r=s-[s]$. It is classical that (we refer to \cite{L})  $f\in W^{s,2}(\mathbb{R}^n)$ if and only if $\partial^\alpha f \in L^2 (\mathbb{R}^n)$ for all $|\alpha| \leqslant [s]$ and (if $r>0$)
$$\frac{\partial^\alpha f(x) - \partial^\alpha f(y)}{|x-y|^{\frac{n}{2}+r}} \in L^2 (\mathbb{R}^n \times \mathbb{R}^n)$$
for all $|\alpha| = [s]$. The norm of $W^{s,2} (\mathbb{R}^n)$ has the following equivalence
\begin{equation}
\| f \| _{W^{s,2}(\mathbb{R}^n)} \approx \sum_{|\alpha| \leqslant [s]} \| \partial^\alpha f \|_{L^2 (\mathbb{R}^n)} +1_{r>0} \sum _{|\alpha| = [s]} \| \frac{\partial^\alpha f(x) - \partial^\alpha f(y)}{|x-y|^{\frac{n}{2}+r}} \|_{L^2(\mathbb{R}^n \times \mathbb{R}^n)}. 
\end{equation}

For $s\in \mathbb{R}$ and $U\subset \mathbb{R}^n$ open, we define the space (following Lebeau \cite{L})
\begin{equation}
W^{s,2} (U) = \{ f\in \mathcal{D}' (U) \; : \; \exists F\in W^{s,2}(\mathbb{R}^n) \; \text{such that} \; F_{\mid U}=f \}. 
\end{equation}
We equip $W^{s,2} (U)$ with the norm
$$\| f \|_{W^{s,2} (U)} = \inf \{\| F \|_{W^{s,2}(\mathbb{R}^n)} \; : \; F_{\mid U}=f  \}.$$

We have the quotient Hilbert structure on $W^{s,2}(U)$ induced by the natural isomorphism between $W^{s,2}(U)$ and
$$
\bigslant{W^{s,2}(\mathbb{R}^n)}{W^{s,2} _{\mathbb{R}^n \setminus U}(\mathbb{R}^n)}.
$$
Since $W^{s,2}_{\mathbb{R}^n \setminus U}(\mathbb{R}^n)$ is a closed subspace of the Hilbert space $W^{s,2}(\mathbb{R}^n)$, it is complemented by its orthogonal

\begin{center}
$W^{s,2}(\mathbb{R}^n) = W^{s,2}_{\mathbb{R}^n \setminus U}(\mathbb{R}^n) \oplus (W^{s,2}_{\mathbb{R}^n \setminus U}(\mathbb{R}^n))^{\perp}$.
\end{center}

This induces an extension operator $\mathcal{T}: W^{s,2}(U) \longrightarrow W^{s,2}(\mathbb{R}^n)$ given by
$$\mathcal{T}(f) = \text{Proj}_{(W^{s,2}_{\mathbb{R}^n \setminus U}(\mathbb{R}^n))^{\perp}}(F)$$
for any choice of $F\in W^{s,2}(\mathbb{R}^n)$ such that $F\mid_U = f$,  
where $$\text{Proj}_{(W^{s,2}_{\mathbb{R}^n \setminus U}(\mathbb{R}^n))^{\perp}}: W^{s,2}(\mathbb{R}^n) \rightarrow (W^{s,2}_{\mathbb{R}^n \setminus U}(\mathbb{R}^n))^{\perp}$$ is the orthogonal projection.

\vspace{0.3cm}

\textbf{The usual definition of Sobolev spaces:} In our definition, we follow \cite{L}. Note that the usual Sobolev spaces $W^{s,2}_\star$ (see Lions and Magenes \cite{J}) are defined as follows:

\begin{itemize}
\item[$\bullet$] If $k\in \mathbb{N}$, then
\begin{center}
$W^{k,2}_\star(U) = \{ f\in L^2(U) \; : \; \partial^\alpha f \in L^2(U) \text{ for all } |\alpha| \leq k \}$.
\end{center}

\item[$\bullet$] If $s\in ]k,k+1[$, then
\begin{center}
$W^{s,2}_\star(U) = [W^{k+1,2}_\star(U), W^{k,2}_\star(U)]_{s-k}$.
\end{center}
And we have
\begin{center}
$W^{s,2}_\star(U) = \{ f\in L^2(U) \; : \; \partial^\alpha f \in W^{s-k,2}_\star(U) \text{ for all } |\alpha| \leq k \}$.
\end{center}

\item[$\bullet$] For $s<0$, $W^{s,2}_\star(U)$ is defined to be the topological dual space of $W^{-s,2}_\star(U)$.
\end{itemize}

\begin{defi}
A bounded open set $U\subset \mathbb{R}^n$ is said to be Lipschitz (or with Lipschitz boundary) if and only if for any $q\in \overline{U}\setminus U$, there exists an orthogonal transformation $\phi : \mathbb{R}^n \rightarrow \mathbb{R}^n$ with $\phi (q)=0$, a Lipschitz function $f: \mathbb{R}^{n-1}\rightarrow \mathbb{R}$, and $r>0$ such that
\begin{center}
$\phi (U\cap B(q,r))=\{(y',y_n)\in B(0,r) \; : \; y_n>f(y') \}$.
\end{center}
\end{defi}

Thanks to the Stein extension theorem (along with the functoriality of interpolations, as discussed in Section 8), for a Lipschitz domain $U\subset \mathbb{R}^n$  and $s\geqslant 0$, we have

\begin{equation}
W^{s,2} (U)=W^{s,2} _{\star}(U).
\end{equation}

In fact, the Stein extension theorem provides even more (we refer to Stein \cite{ST}):

\begin{thm}\label{thm3.2} Take $U\subset \mathbb{R}^n$ open bounded with Lipschitz boundary. Then there is a  linear continuous extension operator $Ext:L^2(U)\to L^2 (\mathbb{R}^n)$ such that for $k\in \mathbb{N}$ the restriction of $Ext$ to $W^{k,2} _{\star} (U)$ induces a linear continuous operator
\begin{center}
$Ext _{W^{k,2} _{\star} (U)} : W^{k,2} _{\star} (U) \mapsto W^{k,2} (\mathbb{R}^n)$.
\end{center}

\end{thm}

\begin{prop}
Let $U \subset \mathbb{R}^n$ be an open bounded with Lipschitz boundary and $s\geqslant 0$. Let $k=[s]$ and $r=s-[s]$. Then $f\in W^{s,2} (U)$ if and only if: 
\begin{itemize}
\item[$(1)$] For all $\left| \alpha \right| \leqslant k$, we have $\partial ^\alpha f \in L^2 (U)$. 
\item[$(2)$] If $r>0$, then

\begin{equation}
 \int \int _{U \times U} \frac{\left|  \partial ^\alpha  f(x) - \partial ^\alpha  f(y)  \right| ^2 }{\left| x-y    \right| ^{n+2r}}dxdy <+\infty.
\end{equation}
 
\end{itemize}

\begin{proof}
This result follows as a classical consequence of $(3.2)$ and $(3.4)$ (as shown by Lemma 3.5 in \cite{L}).
\end{proof}
\end{prop}

\section{The definable site and the main problem.} 
Let $X_{\mathcal{A}}(\mathbb{R}^n)$ be the category of open bounded definable sets in $\mathbb{R}^n$ (the morphisms are the inclusions, or the empty set). We endow $X_{\mathcal{A}}(\mathbb{R}^n)$ with the Grothendieck topology (note that this definition works for more general categories; for full details, see \cite{KS}):

\begin{center}
$S\subset X_{\mathcal{A}}(\mathbb{R}^n)$ is a covering of $U\in X_{\mathcal{A}}(\mathbb{R}^n)$ if and only if $S$ is finite and $U=\bigcup _{O\in S} O$.
\end{center}

We call this the definable site associated to $\mathcal{A}$.

\begin{defi}
A sheaf of $\mathbb{C}$-vector spaces on the site $X_{\mathcal{A}}(\mathbb{R}^n)$ is a contravariant functor
\begin{center}
$\mathcal{F}: X_{\mathcal{A}}(\mathbb{R}^n) \rightarrow $ $\mathbb{C}$-vector spaces,
\end{center}
such that for any $U,V\in X_{\mathcal{A}}(\mathbb{R}^n)$, the sequence
\begin{center}
$0\rightarrow  \mathcal{F} (U\cup V)\rightarrow \mathcal{F}(U) \oplus \mathcal{F} (V) \rightarrow \mathcal{F} (U\cap V)$
\end{center}
is exact.\\
This is equivalent (see Proposition 6.4.1 in \cite{KS}) to saying that if $S=\{O_1,...,O_l\}\subset X_{\mathcal{A}}(\mathbb{R}^n)$ is a cover of $O\in X_{\mathcal{A}}(\mathbb{R}^n)$, and $f_i \in \mathcal{F}(O_i)$ such that 
\begin{equation}
 f_i \mid _{O_i \cap O_j}=f_j \mid _{O_i \cap O_j}  \; \text{for all} \;  i\neq j \; \text{with} \; O_i \cap O_j\neq \emptyset,
\end{equation}
then there is a unique $f\in \mathcal{F}(O)$ such that $f\mid _{O_i}=f_i$ for $i=1,...,l$.\\
If, in addition, we have that for any $U,V \in X_{\mathcal{A}}(\mathbb{R}^n)$ the sequence
\begin{center}
$0\rightarrow \mathcal{F} (U\cup V)\rightarrow \mathcal{F}(U) \oplus \mathcal{F} (V) \rightarrow \mathcal{F} (U\cap V)\rightarrow 0$
\end{center}
is exact, then $\mathcal{F}$ is an \emph{acyclic} sheaf (see Proposition 2.14 in \cite{GS}).
\end{defi}
For a more comprehensive exploration of this topic, we refer to Kashiwara and Schapira \cite{KS}.\\
The following example  was introduced by Kashiwara \cite{KA} to prove the Riemann-Hilbert correspondence:

\begin{example}
We denote by $X_{sa}(\mathbb{R}^n)$ the site associated to the o-minimal structure of globally subanalytic sets. We define the trace of distributions on open bounded subanalytic sets
\begin{center}
$\mathcal{T}: X_{sa}(\mathbb{R}^n) \to \mathbb{R}$-vector spaces,
\end{center}
such that for $U\subset \mathbb{R}^n$ we have
$$\mathcal{T}(U)=\{f\in \mathcal{D}' (U) \; : \; \exists F\in \mathcal{D}' (\mathbb{R}^n) \; \text{such that} \; F_{\mid U}=f   \}. $$
One can show that $f\in \mathcal{T}(U)$ if and only if there are $C>0$, $m\in \mathbb{N}$, and $r\in \mathbb{N}$ such that for any $\phi \in C_{c}^{\infty} (U)$ we have 
$$ \left| <f,\phi> \right| \leqslant C \sum _{\left| \alpha \right| \leqslant m} \sup_{x\in U} \left( \frac{\left| \partial ^\alpha \phi (x) \right|}{d(x,\partial U)^r} \right)    .$$
Then, thanks to the \L{}ojasiewicz inequality
, $\mathcal{T}$ is an acyclic sheaf on the subanalytic site $X_{sa}(\mathbb{R}^n)$ (see \cite[Theorem 3.5]{KA}). 
\end{example} 
\textbf{Problem:} Given $s>0$, is there a sheaf $\mathcal{F}^s$ on the definable site $X_{\mathcal{A}}(\mathbb{R}^n)$ such that for any $U\in X_{\mathcal{A}}(\mathbb{R}^n)$ with Lipschitz boundary, we have 
\begin{center}
 $\mathcal{F}^s(U)=W^{s,2}(U)$ and $H^j(U, \mathcal{F}^s )=0$ for $j>0$?  
\end{center} 

Recall that for any contravariant functor (a presheaf) $\mathcal{F}: X_{\mathcal{A}}(\mathbb{R}^n) \longrightarrow $ $\mathbb{C}$-vector spaces, and $x\in \mathbb{R}^n$, we denote by $\mathcal{F}_x$ the set of germs of sections of $\mathcal{F}$ at $x$  
$$\mathcal{F}_x = \lim _{\begin{tikzcd}
{} \arrow[r, "x\in U"'] & {}
\end{tikzcd}} \mathcal{F}(U) = \bigslant{\sqcup _{x\in U} \mathcal{F}(U)}{\sim}  ,$$
where $f_1 \sim f_2$ if and only if there is a neighborhood $V \subset U_1 \cap U_2$ of $x$ such that $f_1 \mid _V = f_2 \mid _V$. There is a canonical sheaf $\mathcal{F}_+$ associated to $\mathcal{F}$ defined by  
\begin{center}
   $U\in X_{\mathcal{A}} (\mathbb{R}^n) \mapsto \mathcal{F}_+ (U) \subset F(U, \sqcup _{x\in U} \mathcal{F}_x)$, 
   \end{center}   
where $f\in \mathcal{F}_+ (U)$ if for any $x\in U$,  $f(x)\in \mathcal{F}_x$ and there is a neighborhood $V\subset U$ of $x$ and $\phi \in \mathcal{F}(V)$ such that for every $y\in V$, $f(y)$ is a representative of $\phi$ in $\mathcal{F}_y$.

For $s>0$, consider $W^{s,2} _+$ as the canonical sheaf associated to $W^{s,2}$ on the site $X_{\mathcal{A}} (\mathbb{R}^n)$. However, let $U\in X_{\mathcal{A}} (\mathbb{R}^n)$ be with Lipschitz boundary. It can be shown that \( W^{s,2}_+(U) \) can be identified with the space of functions in \( L^2 _{\textbf{loc}} (U) \) that are locally in \( W^{s,2}(U) \). Moreover, every \( C^\infty \) function on \( U \) belongs to \( W^{s,2}_+(U) \); this is not the case for \( W^{s,2}(U) \), which makes the canonical sheafification method unsuitable for our purpose. Our goal is to create a sheaf out of Sobolev spaces while retaining their advantageous properties, as Sobolev spaces work effectively on domains with Lipschitz boundary. For $s<0$, a sheafification in the derived category $D^+ (X_{\mathbb{R}_{\text{an}}} (\mathbb{R} ^n))$ of sheaves on the subanalytic site was provided by Lebeau \cite{L}:
\begin{thm}
For $s<0$, there exists an object $\mathcal{F}^s \in D^+ (X_{sa} (\mathbb{R} ^n))$ such that if $U\subset \mathbb{R}^n$ is a bounded open subanalytic set with Lipschitz boundary, the complex $\mathcal{F}^s (U)$ is concentrated in degree $0$ and is equal to $W^ {s,2}(U)$.
\end{thm} 
\section{The spaces $W^{s,2}$ for $s\in ]-\frac{1}{2}, \frac{1}{2}[$.}
Using the results of Parusi\'nski \cite{Pa2}, it was noticed in \cite{L} that for $s\in ]-\frac{1}{2}, \frac{1}{2}[$, the presheaf $U \mapsto W^{s,2} (U)$ is an acyclic sheaf on the subanalytic site. For the convenience of the reader, we provide detailed explanations of why this is true in the o-minimal case. Let us first recall a classical result on fractional Sobolev spaces (see \cite[Theorem 11.2]{J}). Take $s\in ]0,\frac{1}{2}[$ and $U\subset \mathbb{R}^n$ an open bounded set with Lipschitz boundary. Then there is a $C>0$ such that for any $f\in W^{s,2}(U)$, we have
\begin{equation}
\left\| \frac{f(x)}{d(x,\mathbb{R}^n \setminus U)^s} \right\| _{L^2 (U)}\leqslant C \| f \| _{W^{s,2}(U)} .
\end{equation}

\textbf{Fact}: Fix $s\in ]-\frac{1}{2}, \frac{1}{2} [$ and let $U\in X_{\mathcal{A}} (\mathbb{R}^n)$ be with  Lipschitz boundary. Then the linear operator
\begin{center}
 $1_U : W^{s,2}(\mathbb{R}^n) \longrightarrow W^{s,2}(\mathbb{R}^n)$\\
 $f\mapsto 1_U f$
 \end{center} 
is well defined.
\begin{proof} The case of $s=0$ is obvious. For $0 < s < \frac{1}{2}$, consider $f\in W^{s,2}(\mathbb{R}^n)$ . It is clear that $1_{U}f \in L^2 (\mathbb{R}^n)$, so by $(3.2)$ we need to prove that
\begin{equation}
L=\int \int _{\mathbb{R}^n \times \mathbb{R}^n} \frac{\left|    1_U f(x) -   1_U f(y)  \right| ^2 }{\left| x-y    \right| ^{n+2s}}dxdy <+\infty.
\end{equation}
But
\[
L=\int \int _{U \times U} \frac{\left|    f(x) -   f(y)  \right| ^2 }{\left| x-y    \right| ^{n+2s}}dxdy + 2 \int _U \left| f(x) \right| ^2 \left( \int _{ U ^c} \frac{1 }{\left| x-y    \right| ^{n+2s}}dy \right) dx.
\]
Since $f\in W^{s,2}(\mathbb{R}^n)$, by $(5.1)$ it is enough to prove that
\begin{equation}
d(x,\mathbb{R}^n \setminus U)^{-2s} \lesssim \int _{U} \frac{1}{\left| x-y    \right| ^{n+2s}}dy \lesssim  d(x,\mathbb{R}^n \setminus U)^{-2s},
\end{equation}
where $U\in X_{\mathcal{A}} (\mathbb{R}^n)$ with Lipschitz boundary. Since $\partial U$ is bounded, we can assume that
\begin{equation}
U=\{ (y',y_n) \in \mathbb{R}^n \; :  \; y_n >0   \}.
\end{equation}
A simple computation shows that
\[
d(x,\mathbb{R}^n \setminus U)^{-2s}=\frac{1}{\left| x_n \right| ^{2s}} \lesssim \int _{U} \frac{1}{\left| x-y    \right| ^{n+2s}}dy \lesssim \frac{1}{\left| x_n \right| ^{2s}} =d(x,\mathbb{R}^n \setminus U)^{-2s}.
\]

For $s\in ]-\frac{1}{2} , 0 [$, consider $T\in W^{s,2}(\mathbb{R}^n )$. We have
\begin{center}
$1_U T:W^{-s,2}(\mathbb{R}^n)\longrightarrow \mathbb{C}$\\
$\; \; \; \; \; \; \; \; \; \; \; \; \; \; \; \; \; \; \; \; \; \; \; \; \; \; \; \; \; \; \; f \mapsto <1_U T,f >= <T,1_U f>$.
\end{center}
By the case of $s\in ]0,\frac{1}{2} [$, $1_U T$ is well defined and lies in $W^{-s,2}(\mathbb{R}^n)$.

 \end{proof} 

Let $\mathcal{A} (\mathbb{R}^n)$ denote the algebra generated by the characteristic functions of open bounded definable sets in $\mathbb{R}^n$, that is 
\[
\mathcal{A} (\mathbb{R}^n)=\left\{ \sum _{i\in I} m_i 1_{U_i} \; : \; I \; \text{finite}, \; m_i \in \mathbb{Z}, \; \text{and} \; U_i \in X_{\mathcal{A}}(\mathbb{R}^n)   \right\}.
\]

Then we have Parusi\'nski's result in \cite{Pa2}: 

\begin{thm}\label{thm5.1}
The algebra $\mathcal{A} (\mathbb{R}^n)$ is generated by the characteristic functions of Lipschitz definable domains.
\end{thm} 

Now we explain why for $s\in ]-\frac{1}{2},\frac{1}{2} [$, the presheaf $W^{s,2}$ is an acyclic sheaf on the definable site $X_{\mathcal{A}}(\mathbb{R}^n)$, that is, for any $U,V\in X_{\mathcal{A}}(\mathbb{R}^n)$ the sequence
\[
0\to W^{s,2}(U\cup V) \rightarrow W^{s,2} (U)\oplus W^{s,2} (V) \rightarrow W^{s,2} (U\cap V) \rightarrow 0
\]
is exact.

\begin{proof}
By the definition of $W^{s,2}$, we have the surjectivity of the map 
$$W^{s,2} (U)\oplus W^{s,2} (V) \to W^{s,2} (U\cap V).$$
Take $(f,g)\in W^{s,2} (U)\oplus W^{s,2} (V)$ such that $f_{\mid U\cap V}=g_{\mid U\cap V}$. Take $(\widehat{f} , \widehat{g} )\in (W^{s,2} (\mathbb{R} ^n))^2 $ such that
\[
\widehat{f} _{\mid U}=f \quad \text{and} \quad \widehat{g} _{\mid V}=g.
\]
By the previous fact and Theorem~\ref{thm5.1}, we have $h=1_U \widehat{f} + 1_V \widehat{g} - 1_{U\cap V} \widehat{f} \in W^{s,2} (\mathbb{R}^n)$. Then $h_{\mid U\cup V} \in W^{s,2}(U\cup V)$, $(h_{\mid U\cup V})_{\mid U}=f$, and $(h_{\mid U\cup V})_{\mid V}=g$.
\end{proof}

\section{Construction of the sheaf $\mathcal{F}^k$ on $\mathbb{R}^2$ for $k\in \mathbb{N}$.}
Before we begin, we fix anticlockwise orientation of the plane $\mathbb{R}^2$ generated by the vectors $\overrightarrow{e_1}=(1,0)$ and $\overrightarrow{e_2}=(0,1)$. Given two definable $C^1$-curves $\gamma_1 , \gamma_2 : [0,a[ \longrightarrow \mathbb{R}^2$, and $r>0$ sufficiently small such that $\gamma_1 (0)=\gamma_2 (0)=p_0$, we denote by $R(r,\gamma_1, \gamma_2)$ the open definable subset (see Figure 2):
\[ R(r,\gamma_1, \gamma_2)=\{ P\in \mathbb{R}^2 \; : \; P\in B(p_0,r) \; \text{and} \; P \; \text{is between} \; \gamma_1 \; \text{and} \; \gamma_2  \}. \]
Formally,
\[ P\in R(r,\gamma_1, \gamma_2) \text{ if and only if } \angle(\gamma_1 \cap C(p_0,\| P \|),\overrightarrow{e_1} )< \angle(\overline{p_0 P},\overrightarrow{e_1})  <\angle(\gamma_2 \cap C(p_0,\| P \|), \overrightarrow{e_1}). \]
Here,
\[ C(p_0 , \| P \|)=\{x\in \mathbb{R}^2 \; : \; \| x - p_0 \| = \| P \|    \}. \]

If we parameterize $\gamma_1$ and $\gamma_2$ by the distance to $p_0$ (assume that $p_0=0$, which is always possible up to a translation):

\begin{center}
$\gamma_1 (t)=te^{i\theta _1 (t)}$ and $\gamma_2 (t)=te^{i\theta _2 (t)}$ with $t\in [0,r[$ and $0<\theta _2 (t)-\theta _1 (t)<2\pi$.
\end{center}

Then,
\begin{center}
$R(r,\gamma_1, \gamma_2)=\{te^{i \theta} \; : \; t\in ]0,r[ \; \text{and} \; \theta _1 (t)<\theta <\theta _2 (t)     \}$.
\end{center}

\begin{rmk}
We can always choose $r$ to be small enough such that $R(r,\gamma_1, \gamma_2)$ is connected and the circle $C(p_0 , r')$  (for $r'<r$) is transverse to $\gamma_1$ and $\gamma_2$ at the intersection points (which consist of only two points).
\end{rmk}

 \begin{figure}[H]
  \centering
  \includegraphics[scale=0.4]{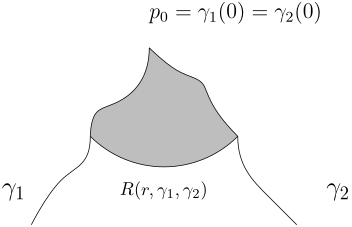}
  \caption{The domain $R(r,\gamma _1, \gamma _2)$.}   
\end{figure}

\subsection{The local nature of open definable sets in $\mathbb{R}^2$.} Let $U$ be a bounded connected open definable subset of $\mathbb{R}^2$. By choosing a cell decomposition of $\mathbb{R}^2$ compatible with $U$ and $\partial U$, we can prove that for any $p_0 \in \partial U$ there is $r>0$ such that we have one of the following cases: 

\begin{itemize}
\item[$(C_1)$:] \textbf{Punctured disk.} $B_r (p_0) \cap U= B_r (p_0)\setminus \{ p_0  \}$.
\begin{figure}[H]
  \centering
  \includegraphics[scale=0.3]{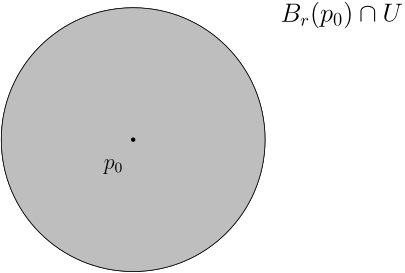}
  \caption{The $(C_1)$ case.}   
\end{figure}
\item[$(C_2)$:] \textbf{Sector.} There are two definable $C^1$-curves $\gamma _1 , \gamma _2 : [0,a[ \longrightarrow \mathbb{R}^2$ such that $\gamma _1 (0) = \gamma _2 (0)=p_0$, $\angle(\gamma _1 '(0),  \gamma _2 ' (0))\neq 0, 2\pi$, and

\begin{center}
$B_r (p_0) \cap U=R(r,\gamma _1 , \gamma _2)$.
\end{center}
\begin{figure}[H]
  \centering
  \includegraphics[scale=0.4]{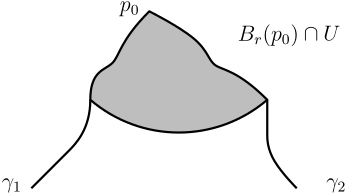}
  \caption{The $(C_2)$ case.}   
\end{figure}
\item[ $(C_3)$:] \textbf{Cusp.} There are two definable $C^1$-curves $\gamma _1 , \gamma _2 : [0,a[ \longrightarrow \mathbb{R}^2$ such that $\gamma _1 (0) = \gamma _2 (0)=p_0$, $\angle(\gamma _1 '(0),  \gamma _2 ' (0))=0$, and

\begin{center}
$B_r (p_0) \cap U=R(r,\gamma _1 , \gamma _2)$.
\end{center}
\begin{figure}[H]
  \centering
  \includegraphics[scale=0.4]{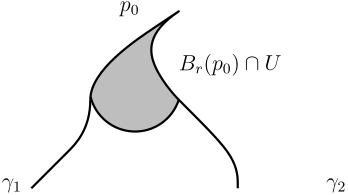}
  \caption{The $(C_3)$ case.}   
\end{figure}
\item[ $(C_4)$:] \textbf{Cusp complement.} There are two definable $C^1$-curves $\gamma _1 , \gamma _2 : [0,a[ \longrightarrow \mathbb{R}^2$ such that $\gamma _1 (0) = \gamma _2 (0)=p_0$, $\angle(\gamma _1 '(0),  \gamma _2 ' (0))=2\pi$, and
\begin{center}
$B_r (p_0) \cap U=R(r,\gamma _1 , \gamma _2)$.
\end{center}
\begin{figure}[H]
  \centering
  \includegraphics[scale=0.4]{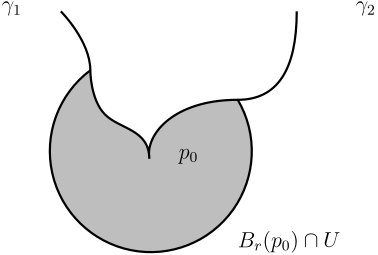}
  \caption{The $(C_4)$ case.}   
\end{figure}
\item[ $(C_5)$:] \textbf{Arc complement.} There exists a definable $C^1$-curve $\gamma : [0,a[ \longrightarrow \mathbb{R}^2$ such that $\gamma (0)=p_0$ and 

\begin{center}
$B_r (p_0) \cap U=B_r(p_0) \setminus Im(\gamma)$.
\end{center}
\begin{figure}[H]
  \centering
  \includegraphics[scale=0.4]{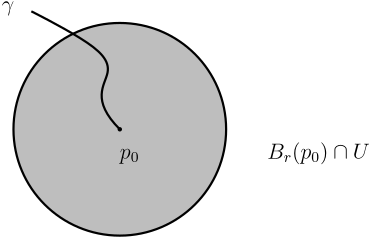}
  \caption{The $(C_5)$ case.}   
\end{figure}
\item[$(C_6)$] $B_r (p_0) \cap U$ is a disjoint union of copies of open sets  like  $C_2$, $C_3$, and $C_4$.
\begin{figure}[H]
  \centering
  \includegraphics[scale=0.4]{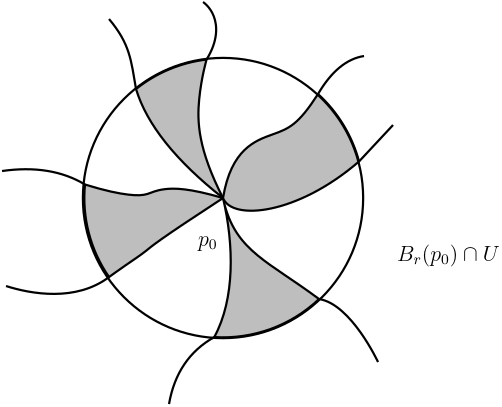}
  \caption{The $(C_6)$ case.}   
\end{figure}
\end{itemize}

\subsection{Local definition of the sheaf $\mathcal{F}^k$.}

\begin{lem}\label{Proposition2.2} Let $U$, $V$ be two Lipschitz definable bounded open subsets of $\mathbb{R}^n$ such that $U\cup V$ and $U\cap V$ are Lipschitz. For any $s\in \mathbb{R}_+$, the sequence of Hilbert spaces 
\begin{center}
$0\to W^{s,2} (U\cup V)\to W^{s,2} (U)\oplus W^{s,2} (V) \to W^{s,2} (U\cap V) \to 0$
\end{center}
is exact.
\begin{proof}
See \cite{L} for the proof (or see Section 8 for a categorical proof).
\end{proof}

\end{lem}

\begin{rmk}\label{rem_s_in_N}
For $s\in \mathbb{N}$, the requirement for $U\cap V$ to be Lipschitz in the statement of Lemma~\ref{Proposition2.2} is not necessary.
\end{rmk}
\begin{proof}
Take $s=k\in \mathbb{N}$. By $(3.4)$, for $\Omega = U \cup V \subset \mathbb{R}^n$, we have
\begin{center}
$W^{k,2}(\Omega)=\{f\in L^2(\Omega) \; : \; \forall \alpha \in \mathbb{N}^n \; \text{such} \; \text{that}  \left| \alpha \right| \leqslant k, \; \partial ^{\alpha}f \in L^2 (\Omega) \;   \}$,
\end{center}
where $\partial ^{\alpha}f$ is the distributional derivative of $f$. The Hilbert structure of $W^{k,2}(\Omega)$ is given by $$\| f \| _{W^{k,2}(\Omega)} ^2=\sum _{\mid \alpha \mid \leqslant k} \| \partial ^{\alpha}f \| _{L^2(\Omega)} ^2 .$$
Now, consider $(f,g)\in W^{k,2} (U)\oplus W^{k,2} (V)$ such that $f\mid _{U\cap V}=g \mid _{U\cap V}$. There exists $H\in L^2(U \cup V)$ such that $H\mid _{U}=f\in W^{k,2}(U)$ and $H\mid _{V}=g\in W^{k,2}(V)$. We aim to show that for any $\alpha \in \mathbb{N}^n$ with $\left| \alpha \right| \leqslant k$, there exists $h_\alpha \in L^2 (U \cup V)$ such that $\partial ^{\alpha }H=h_\alpha$ (in the distributional sense).

Let $(\varphi _U , \varphi _V)$ be a partition of unity associated to $(U,V)$. For any $\phi \in C_c ^{\infty}(U \cup V)$, we have
\begin{center}
	\begin{align*}
	<\partial ^\alpha H , \phi > &= <\partial ^\alpha H , \varphi _U \phi >+<\partial ^\alpha H , \varphi _V \phi >\\
	&= (-1)^{\mid \alpha \mid}\int _U H\partial ^\alpha (\varphi _U \phi) + (-1)^{\mid \alpha \mid}\int _V H\partial ^\alpha (\varphi _V \phi)\\
	 &= (-1)^{\mid \alpha \mid}\int _U f\partial ^\alpha (\varphi _U \phi) + (-1)^{\mid \alpha \mid}\int _V g\partial ^\alpha (\varphi _V \phi)\\
	 &= \int _U \partial ^\alpha f (\varphi _U \phi) + \int _V \partial ^\alpha g (\varphi _V \phi)\\
	 &= \int _{U\cup V} (\varphi _U \partial ^{\alpha}f +  \varphi _V \partial ^{\alpha}g )\phi \\
	 &= \int _{U\cup V} h_\alpha \phi.
	\end{align*}
\end{center}

Here, $h_\alpha := \varphi _U \partial ^{\alpha}f +  \varphi _V \partial ^{\alpha}g \in L^2 (U \cup V)$, which completes the proof.
\end{proof}

From now on, we consider $k\in \mathbb{N}$. Let $U$ be a connected open definable bounded subset of $\mathbb{R}^2$. We define the $\mathbb{C}$-vector space $\widehat{\mathcal{F}}^k (U)$ in the following special cases:

\begin{itemize}
\item[$(C_1)$] If $U=B_r (p_0)\setminus \{ p_0  \}$, we can assume $p_0=(0,0)$ and $r=1$. In this case, we can decompose $U=U_1 \cup U_2$, where 
\begin{center}
$U_1=\{(x,y)\in U: \; y > x \; \text{or} \; y<-x   \}$ and $U_2=\{(x,y)\in U: \; y> -x \; \text{or} \; y< x   \}$.
\end{center}

We have the sequence
\begin{center}
\begin{tikzcd}
0 \arrow[r] & W^{k,2}(U) \arrow[r, "d_0"] & W^{k,2}(U_1) \oplus W^{k,2} (U_2) \arrow[r, "d_1"] &  W^{k,2} (U_1 \cap U_2)
\end{tikzcd}
\end{center}
It follows from Lemma~\ref{Proposition2.2} that
\begin{center}
 $Ker(d_1)=\{f\in L^2(U)\; : \; f\mid _{L}\in W^{k,2}(L) \; \text{for any} \; L \; \text{Lipschitz in} \; U  \} = W^{k,2} _{\star} (U) $.
\end{center}
But we have a fact (we refer to Exercise 11.9 in \cite{Geov}) about Sobolev spaces:\\
\textbf{Fact:} Take $\Omega \subset \mathbb{R}^n$ open and $W\subset \Omega$ such that $\mathcal{H} ^{n-1} (W)=0$, where $\mathcal{H} ^{n-1}$ is the $(n-1)$-Hausdorff measure on $\mathbb{R}^n$. Then we have
\begin{center}
 $W^{k,2} _{\star} (\Omega \setminus W)=W^{k,2} _{\star} (\Omega)$.
 \end{center} 

That gives $$W^{k,2} _{\star} (U)=W^{k,2} _{\star} (B_r (p_0))=W^{k,2}  (B_r (p_0)).$$
So, this means that the sequence
\begin{center}
\begin{tikzcd}
0 \arrow[r] & W^{k,2}(U) \arrow[r, "d_0"] & W^{k,2}(U_1) \oplus W^{k,2} (U_2) \arrow[r, "d_1"] &  W^{k,2} (U_1 \cap U_2)
\end{tikzcd}
\end{center}
is exact. Therefore, we can define $\widehat{\mathcal{F}}^k(U)$ by 
\begin{center}
$\widehat{\mathcal{F}}^k(U):=W^{k,2}(U).$
\end{center}

\item[$(C_2)$] If $U$ is connected with Lipschitz boundary, then we define $\widehat{\mathcal{F}}^k (U)=W^{k,2} (U)$.
\item[$(C_3)$] If $U$ is a cusp, meaning that there are  $r>0$ and two definable $C^1$-curves $\gamma _1 , \gamma _2 : [0,a[ \longrightarrow \mathbb{R}^2$ such that $\gamma _1 (0) = \gamma _2 (0)$, $\angle(\gamma _1 '(0),  \gamma _2 ' (0))=0$, and
$$U=R(r,\gamma _1 , \gamma _2).$$
Then we define: $\widehat{\mathcal{F}}^k (U)=W^{k,2} (U)$.
\item[$(C_4)$] Let $U$ be the complement of a cusp. This means that there are  $r>0$ and two definable $C^1$-curves $\gamma _1 , \gamma _2 : [0,a] \longrightarrow \mathbb{R}^2$ such that $\gamma _1 (0) = \gamma _2 (0)=p_0$, $\angle(\gamma _1 '(0), \gamma _2 '(0) )=2\pi$, and

\begin{center}
$U=R(r,\gamma _1, \gamma _2)$.
\end{center}

Take $\gamma _3 , \gamma _4: [0,a[\longrightarrow \mathbb{R}^2$ such that $\gamma _3 (0)=\gamma _4 (0)=p_0$, $\angle(\gamma _1 ' (0), \gamma _{3} ' (0)) \notin \{0, 2\pi \}$,  $\angle(\gamma _4 ' (0), \gamma _{2} ' (0)) \notin \{0, 2\pi \}$, and $\angle(\gamma _3 ' (0), \gamma _{4} ' (0)) \neq 0$.

In this case, the sequence 

\begin{center}
$0\to W^{k,2} (U)\to W^{k,2} (R(r,\gamma _1, \gamma _4))\oplus W^{k,2} (R(r,\gamma _3, \gamma _2)) \to W^{k,2} (R(r,\gamma _3, \gamma _4)) $
\end{center}
is not exact in general. 
 
\begin{example} Assume that $s>\frac{1+\sqrt{5}}{2}$. Then we have the continuous embedding $W^{s,2}(\mathbb{R}^2) \hookrightarrow C^{0, s-1}(\mathbb{R}^2)$ (see \cite[Theorem 11.32]{Geov}). Take $U,V\in X_{\mathcal{A}}(\mathbb{R}^2)$ defined by 
\begin{center}
 $U=(]-1,1[\times ]-1,0[)\cup (]-1,0[\times ]-1,1[)$,
 \end{center} 
and 

\begin{center}
$V=(]-1,0[\times ]-1,1[) \cup  \{ (x,y)\; : \; 0\leqslant x <1 \; \text{and} \; x^{s+1} < y <1  \}$.
\end{center}

Define $F\in L^2(U\cup V)$ by $F\mid _U =0$, $F(x,y)=x^{s}$ for $x\in [0,1[$ and $x^{s+1}<y<1$. It is clear that $F\mid _U \in W^{s,2}(U)$ and $F\mid _V \in W^{s,2}(V)$ but $F\notin W^{s,2} (U\cup V)$, because if $F\in W^{s,2} (U\cup V)$ then there will be a $C^{0,s-1}$ extension $\widehat{F}$ of $F$ to $\mathbb{R}^2$, but this cannot be true because
$$
\frac{ \lvert \widehat{F}(x, x^{s+1}) - \widehat{F}(x, 0) \rvert }
     { \lvert x^{s+1} - 0 \rvert^{s-1} }
= x^{s - (s+1)(s-1)} = x^{-s^2 + s + 1}
$$
which is not bounded near $x=0$.
\begin{flushright}
$\square$
\end{flushright}

\end{example} 

\begin{question} 
What happens in this case of $s\in [\frac{1}{2},\frac{1+\sqrt{5}}{2}]$ ? Is the sequence 
\begin{center}
$0\to W^{s,2} (U)\to W^{s,2} (R(r,\gamma _1, \gamma _4))\oplus W^{s,2} (R(r,\gamma _3, \gamma _2)) \to W^{s,2} (R(r,\gamma _3, \gamma _4))$
\end{center}
exact?
\end{question}
Now we define $\widehat{\mathcal{F}}^k (U)$  to be the kernel of the map
\begin{center}
$J:W^{k,2} (R(r,\gamma _1, \gamma _4))\oplus W^{k,2} (R(r,\gamma _3, \gamma _2)) \rightarrow W^{k,2} (R(r,\gamma _3, \gamma _4))$.
\end{center}
We use the notation 
\begin{center}
$\widehat{\mathcal{F}}^k (U)=Ker(J)=K(\gamma _3 , \gamma _4)$.
\end{center}
We need to prove that $K(\gamma _3 , \gamma _4)$ doesn't depend on $\gamma _3$ and $\gamma _4$, but only on $U$. Take $\alpha, \beta :[0,a[\longrightarrow \mathbb{R}^2$ two definable curves that satisfy the same conditions as $\gamma _3$ and $\gamma _4$. Let's prove that 
\begin{center}
$K(\gamma _3 , \gamma _4)=K(\alpha , \beta)$.
\end{center}
It is sufficient to show that $K(\gamma _3 , \gamma _4)\subset K(\alpha , \beta)$ (the reverse inclusion follows in the same way).
We can identify $K(\gamma _3 , \gamma _4)$ and $K(\alpha , \beta)$ with the spaces 
\begin{center}
$K(\gamma _3 , \gamma _4)=\{f\in \mathcal{D}'(U) \; : \; f_{\mid R(r,\gamma _1, \gamma _4)}\in W^{k,2} (R(r,\gamma _1, \gamma _4)) \; \text{and} \;  f_{\mid R(r,\gamma _3, \gamma _2)}\in W^{k,2} (R(r,\gamma _3, \gamma _2))    \}$
\end{center}

\begin{center}
$K(\alpha , \beta)=\{f\in \mathcal{D}'(U) \; : \; f_{\mid R(r,\gamma _1, \beta)}\in W^{k,2} (R(r,\gamma _1, \beta)) \; \text{and} \;  f_{\mid R(r,\alpha, \gamma _2)}\in W^{k,2} (R(r,\alpha, \gamma _2))    \}$.
\end{center}

We can distinguish four possible cases:

\begin{itemize}
\item[Case 1:] $Im(\alpha) \setminus \{ p_0 \} \subset R(r,\gamma _3 , \gamma _4)$ and $Im(\beta) \setminus \{ p_0 \} \subset R(r,\gamma _3 , \gamma _4)$. 
\item[Case 2:] $Im(\alpha) \setminus \{ p_0 \} \subset R(r,\gamma _3 , \gamma _4)$ and $Im(\beta) \setminus \{ p_0 \} \subset R(r,\gamma _4 , \gamma _2)$.
         \item[Case 3:]  $Im(\alpha) \setminus \{ p_0 \} \subset R(r,\gamma _1 , \gamma _3)$ and $Im(\beta) \setminus \{ p_0 \} \subset R(r,\gamma _3 , \gamma _4)$.
         \item[Case 4:]  $Im(\alpha) \setminus \{ p_0 \} \subset R(r,\gamma _1 , \gamma _3)$ and $Im(\beta) \setminus \{ p_0 \} \subset R(r,\gamma _1 , \gamma _3)$.
\end{itemize}

The first case is obvious, because in this case we have $R(r,\gamma _1 , \beta) \subset R(r,\gamma _1 , \gamma _4)$ and $R(r,\alpha,\gamma _2) \subset R(r,\gamma _3 , \gamma _2)$. The cases 3 and 4 can be proven using the same computation as Case 2.\\
\textbf{Proof in Case 2:} Take $f\in K(\gamma _3 , \gamma _4)$. In this case, since $R(r,\alpha , \gamma _2)\subset R(r,\gamma _3 , \gamma _2)$, we have $f_{\mid R(r,\alpha , \gamma _2)}\in W^{k,2} (R(r,\alpha , \gamma _2))$. Now let's prove that 

\begin{center}
$f_{\mid R(r,\gamma _1 , \beta)}\in W^{k,2} (R(r,\gamma _1 , \beta))$.
\end{center}

Take $c:[0,a[\longrightarrow \mathbb{R}^2$ a definable curve such that $c(0)=p_0$, $\angle(\gamma _1 '(0),c'(0))>0$, $\angle(c'(0), \gamma _2 '(0))>0$, $\angle(\beta '(0),c'(0))>0$, and $Im(c)\subset R(r,\beta , \gamma _2) $. We can see that $f_{\mid R(r,\gamma _1 ,\gamma _4)}\in W^{k,2} (R(r,\gamma _1 ,\gamma _4))$ and $f_{\mid R(r,\gamma _3 ,c)}\in W^{k,2} (R(r,\gamma _3 ,c))$ (note that $R(r,\gamma _3 ,c)\subset R(r,\gamma _3 ,\gamma _2)$). Now, by Lemma~\ref{Proposition2.2} the sequence

\begin{center}
$0\to W^{k,2} (R(r,\gamma _1, c))\to W^{k,2} (R(r,\gamma _1 ,\gamma _4))\oplus W^{k,2} (R(r,\gamma _3 ,c)) \to W^{k,2} (R(r,\gamma _3 ,\gamma _4))$
\end{center}

is exact. Hence, $f_{\mid R(r,\gamma _1, c)} \in W^{k,2} (R(r,\gamma _1, c))$, which implies $f_{\mid R(r,\gamma _1, \beta)} \in W^{k,2} (R(r,\gamma _1, \beta))$. 
\begin{flushright}
$\square$
\end{flushright}

\item[$(C_5)$] If there exists a definable $C^1$-curve $\gamma : [0,a[ \longrightarrow \mathbb{R}^2$ such that $\gamma (0)=p_0$ and 

\begin{center}
$U=B_r(p_0) \setminus Im(\gamma)$,
\end{center}

take $\gamma _1, \gamma _2 :[0,a[ \longrightarrow \mathbb{R}^2$ two definable  $C^1$-curves such that  $\angle(\gamma _1 '(0),\gamma _2 ' (0))\neq 2\pi $, $\angle(\gamma _1 '(0),\gamma  ' (0))\neq 0 $ and $\angle(\gamma '(0),\gamma _2 ' (0))\neq 0 $. By Sobolev embeddings and continuity reasons, we can find an example such that the sequence 

\begin{center}
$0\to W^{k,2} (U)\to W^{k,2} (R(r,\gamma , \gamma _2))\oplus W^{k,2} (R(r,\gamma _1 , \gamma )) \to W^{k,2} (R(r,\gamma _1, \gamma _2))$
\end{center}

is not exact. 

\begin{example} Assume that $k>1$. So we have an embedding $W^{k,2}(\mathbb{R}^2)\hookrightarrow C^0 (\mathbb{R}^2)$. Take $U,V\in X_{\mathcal{A}}(\mathbb{R}^2)$ defined by 

\begin{center}
 $U=(]-1,1[\times ]-1,0[)\cup (]-1,0[\times ]-1,1[)$,
 \end{center}
 and

 \begin{center}
 $V=(]-1,0[\times ]-1,1[)\cup (]-1,1[\times ]0,1[)$.
 \end{center} 
Define $F\in L^2(U\cup V)$ by $F\mid _{ U \cup (\{0\} \times ]0,1[) } =0$ and $F(x,y)=e^{-\frac{1}{x^{2}}}$ for $0<x<1$ and $0<y<1$. Then $F\mid _U \in W^{k,2}(U)$ and $F\mid _V \in W^{k,2}(V)$, but $F\notin W^{k,2}(U\cup V)$, as it cannot be extended to a continuous function on $\mathbb{R}^2$.

\begin{flushright}
$\square$
\end{flushright}
\end{example}

\begin{question}
What happens in this case if we replace $k$ with $s\in [\frac{1}{2} , 1]$ ? Is the sequence
\begin{center}
$0\to W^{s,2} (U)\to W^{s,2} (R(r,\gamma , \gamma _2))\oplus W^{s,2} (R(r,\gamma _1 , \gamma )) \to W^{s,2} (R(r,\gamma _1, \gamma _2))$
\end{center}
exact?
\end{question}

So this motivates us to define $\widehat{\mathcal{F}}^k(U)$ to be the kernel of the map 

\begin{center}
$J:W^{k,2} (R(r,\gamma , \gamma _2))\oplus W^{k,2} (R(r,\gamma _1 , \gamma )) \to W^{k,2} (R(r,\gamma _1, \gamma _2))$.
\end{center}

That is,
\begin{center}
$\widehat{\mathcal{F}}^k(U)=Ker(J)=K(\gamma _1, \gamma _2)$.
\end{center}

Applying the same techniques as in the previous case, we can show that $K(\gamma _1, \gamma _2)$ does not depend on $\gamma _1$ and  $\gamma _2$. 
\begin{rmk}
Note that this is a special case of the previous case.
\end{rmk}

\item[$(C_6)$] If $U$ is as described in case $(C_6)$, we define $\widehat{\mathcal{F}}^k(U)$ to be the direct sum of the sections of $\widehat{\mathcal{F}}^k$ on the connected components of $U \cap B_r(p_0)$.

\end{itemize}

\subsection{The global definition of $\mathcal{F}^k$ on the site $X_{\mathcal{A}}(\mathbb{R}^2)$  .}

Take $k\in \mathbb{N}$. For every $U\in X_{\mathcal{A}}(\mathbb{R}^2)$, we define $\mathcal{F}^k  (U)$ by

\begin{center}
$\mathcal{F}^k  (U):=\{f\in W^{k,2} _{loc}(U) \; : \; \text{for  each} \; x\in \partial U, \; \exists r>0 \; \text{such  that} \; B(x,r)\cap U \in \{ C_1 ,..., C_6 \} \; \text{and} \; f\mid _{B(x,r)\cap U}\in \widehat{\mathcal{F}}^{k}(B(x,r)\cap U) \;    \}$.
\end{center}

\textbf{Claim:}  $\mathcal{F}^k$ is a sheaf on the  site $X_{\mathcal{A}}(\mathbb{R}^2)$.
\begin{proof}
We need to prove that for $U\in X_{\mathcal{A}}(\mathbb{R}^2)$ and $V\in X_{\mathcal{A}}(\mathbb{R}^2)$, the 
 sequence

 \begin{center}
  $0 \to \mathcal{F}^k (U\cup V) \to \mathcal{F}^k (U) \oplus \mathcal{F}^k (V) \to \mathcal{F}^k (U\cap V)$
  \end{center} 

is exact. It is enough to prove that if $f\in W^{k,2} _{loc}(U\cup V)$ (or even $\mathcal{D}'(U\cup V)$) such that $f \mid _{U} \in \mathcal{F}^k (U)$ and $f \mid _{V} \in \mathcal{F}^k (V) $, then one has $f\in \mathcal{F}^k (U\cup V)$. It is also enough to assume that $f$ is supported in a small neighborhood of a given point $x\in \overline{U\cup V}$ (if $(\phi _ i)$ is a partition of unity such that $\sum _{i} \phi _i =1$ near $\overline{U \cup V}$, then clearly $f=\sum _i \phi _i f$), and more precisely of a given singular point $x\in \partial U \cap \partial V$ such that $\partial U$ and $\partial V$ have different germs at $x$. \\

So take $x \in \partial U \cap \partial V$ such that $(\partial U, x)\neq (\partial V, x)$ (and also no inclusion between the two germs).

\begin{itemize}
\item \textbf{Case(A):} Assume that $U$, $V$, and $U \cup V$ do not fall under case $C_1$.\\

\textbf{Step1:} We assume that $U$, $V$, and  $U\cup V$ are locally connected near $x$; that is, the intersections of $U$, $V$, and  $U\cup V$ 
 with a sufficiently small ball centered at $x$ are connected sets. There is $r>0$ such that $B(x,r)\cap U$, $B(x,r)\cap V$, $B(x,r)\cap (U\cup V)$, $B(x,r)\cap (U\cap V) \in \{C_2 ,...,C_5 \}$, hence there are definable curves $\gamma _i :[0,a[\longrightarrow \mathbb{R}^2$ ($i=1,...,4$) such that $\gamma _i (0)=x$ and 

\begin{center}
$U\cap B(x,r)=R(r,\gamma _1 , \gamma _3)$, $(U\cap V)\cap B(x,r)=R(r,\gamma _2 , \gamma _3)$, $V \cap B(x,r)=R(r,\gamma _3 , \gamma _4)$, and $(U\cup V)\cap B(x,r)=R(r,\gamma _2 , \gamma _4)$.
\end{center}

By the definition of $\mathcal{F}^k$ and assuming that $f$ is supported in $(U\cup V )\cap B(x,r)$,  it is enough to prove that $f\mid _{(U\cup V)\cap B(x,r)} \in \widehat{\mathcal{F}}^k ((U\cup V)\cap B(x,r))= \widehat{\mathcal{F}}^k (R(r,\gamma _1 , \gamma _4))$ knowing that $f\mid _{U\cap B(x,r)} \in \widehat{\mathcal{F}}^k (U\cap B(x,r))$ and $f\mid _{ V\cap B(x,r)} \in \widehat{\mathcal{F}}^k ( V\cap B(x,r))$. We will discuss several cases for this:   

\begin{itemize}

\item \textbf{case(1)} \textbf{ $\angle(\gamma _1 ' (0),\gamma _4 ' (0))=0$ :} In this case, $U$, $V$, $U\cap V$ and $U\cup V$ are cusps near $x$. So we can find $U'$ and $V'$ Lipschitz such that $U' \cup V'$ is Lipschitz, $U\cap B(r,x) \subset U'$, $V\cap B(r,x)\subset V'$, and $U' \cap V'=(U \cap V) \cap B(r,x)$.  In this case, we have

\begin{center}
$\widehat{\mathcal{F}}^k ((U\cup V )\cap B(x,r))=W^{k,2}((U\cup V )\cap B(x,r))$\\
$\widehat{\mathcal{F}}^k (U \cap B(x,r))=W^{k,2}(U \cap B(x,r))$\\
$\widehat{\mathcal{F}}^k (V \cap B(x,r))=W^{k,2}( V \cap B(x,r))$

\end{center}

Take $f_{U'} \in W^{k,2}(U')$ an extension of $f\mid  _{U\cap B(x,r)}$ and $f_{V'} \in W^{k,2}(V')$ an extension of $f\mid _{V \cap B(x,r)}$, and define $F \in \mathcal{D}' (U' \cup V')$ by gluing $f_{U'}$ and $f_{V'}$. By Lemma~\ref{Proposition2.2} we have that $F\in W^{k,2}(U' \cup V')$ and since $F\mid _{(U\cup V)\cap B(x,r)}=f \mid _{(U\cup V)\cap B(x,r)}$,  $f\in W^{k,2}((U\cup V)\cap B(x,r))=\widehat{\mathcal{F}}^k((U\cup V)\cap B(x,r))$.

\item \textbf{case(2)}  \textbf{ $\angle(\gamma _1 ' (0),\gamma _4 ' (0))\neq 0, 2\pi$ :} In this case, either $U$ is Lipschitz or $V$ is Lipschitz in a neighborhood of $x$. If both are Lipschitz, then the proof follows from Lemma~\ref{Proposition2.2}. Let's assume that $U$ is not Lipschitz. In this case, we can find $U'$ Lipschitz such that $U' \cup V$ is Lipschitz, $U\cap B(r,x) \subset U'$, and $U' \cap V=(U \cap V) \cap B(r,x)$.  As in the previous case, we have

\begin{center}
$\widehat{\mathcal{F}}^k ((U\cup V )\cap B(x,r))=W^{k,2}((U\cup V )\cap B(x,r))$\\
$\widehat{\mathcal{F}}^k (U \cap B(x,r))=W^{k,2}(U \cap B(x,r))$\\
$\widehat{\mathcal{F}}^k (V \cap B(x,r))=W^{k,2}( V \cap B(x,r))$

\end{center}

Take $f_{U'} \in W^{k,2}(U')$ an extension of $f\mid  _{U\cap B(x,r)}$, and define $F \in \mathcal{D}' (U' \cup V)$ by gluing $f_{U'}$ and $f\mid _{V}$. By Lemma~\ref{Proposition2.2} we have that $F\in W^{k,2}(U' \cup V)$ and since $F\mid _{(U\cup V)\cap B(x,r)}=f \mid _{(U\cup V)\cap B(x,r)}$,  $f\in W^{k,2}((U\cup V)\cap B(x,r))=\widehat{\mathcal{F}}^k((U\cup V)\cap B(x,r))$.

\item \textbf{case(3)}  \textbf{ $\angle(\gamma _1 ' (0),\gamma _4 ' (0))= 2\pi$ :} 

\begin{itemize}
\item \textbf{Subcase3.1:} If $U\cap B(x,r)$ and $V\cap B(x,r)$ are Lipschitz, then we have by definition that

\begin{center}
$\widehat{\mathcal{F}}^k ((U\cup V)\cap B(x,r))=K(\gamma _{2}, \gamma _{3})$
\end{center}
And this gives that $f\mid _{(U\cup V)\cap B(x,r)} \in \widehat{\mathcal{F}}^k((U\cup V)\cap B(x,r)) $. 
 
\item \textbf{Subcase3.2:} If $\angle(\gamma _{1}' (0), \gamma _{3}'(0))=2\pi$ and $\angle(\gamma _{2}' (0), \gamma _{4}'(0))=2\pi$, then in this case we can find $\alpha$ and $\beta$ in $(U\cap V)\cap B(x,r)$ with a starting point $x$, such that

\begin{center}
 $\widehat{\mathcal{F}}^k ((U\cup V)\cap B(x,r))=K(\alpha, \beta)$
 \end{center} 
And since $f\mid _{R(r,\gamma _1 , \beta)}\in W^{k,2}(R(r,\gamma _1 , \beta))$ and $f\mid _{R(r, \alpha , \gamma _4)}\in W^{k,2}(R(r, \alpha , \gamma _4))$,  we have $f\mid _{(U\cup V)\cap B(x,r)} \in \widehat{\mathcal{F}}^k((U\cup V)\cap B(x,r)) $.

\item \textbf{Subcase3.3:} If  $\angle(\gamma _{1}' (0), \gamma _{3}'(0))=0$ and $\angle(\gamma _{2}' (0), \gamma _{4}'(0))=2\pi$, then in this case we can find  $\alpha , \; \beta :[0,a[\longrightarrow \mathbb{R}^2$ such that $Im(\beta), \; Im(\alpha)\subset  V\cap B(x,r)$, and 

\begin{center}
$\widehat{\mathcal{F}}^k ((U\cup V)\cap B(x,r))=K(\alpha, \beta)$
\end{center}
 
we have that $f\mid _{R(r, \alpha , \gamma _4)}\in W^{k,2}(R(r, \alpha , \gamma _4))$, and by applying \textbf{case (2)} on $R(r,\gamma _1, \gamma _3)$, $R(r,\gamma _2, \beta)$, we deduce that also $f\mid _{R(r,\gamma _1 , \beta)}\in W^{k,2}(R(r,\gamma _1 , \beta))$, hence we have that $f\mid _{(U\cup V)\cap B(x,r)} \in \widehat{\mathcal{F}}^k((U\cup V)\cap B(x,r)) $.

\item \textbf{Subcase3.4:} If  $\angle(\gamma _{1}' (0), \gamma _{3}'(0))=2\pi$ and $\angle(\gamma _{2}' (0), \gamma _{4}'(0))=0$, then it is the symmetry statement of Subcase 3.3. 
 
\end{itemize}

\begin{rmk} Note that the case where
 $\gamma _1 (t)= \gamma _4 (t)$ is included in the \textbf{case(3)}.
\end{rmk}

\end{itemize}

\textbf{Step2:} We don't assume here the local connectivity of $U$, $V$, and $U\cup V$.

In this case, there are a finite number of definable curves ( with beginning point $x$ )
 $\gamma _1,\lambda _1...,\gamma _m, \lambda _m :[0,a[\longrightarrow \mathbb{R}^2$, $\alpha _1,\beta _1...,\alpha _l, \beta _l :[0,a[\longrightarrow \mathbb{R}^2$ 
such that  

\begin{center}
$B(x,r)\cap U = \sqcup _i R(r,\gamma _i , \lambda _i )$ and $B(x,r)\cap V = \sqcup _i R(r,\alpha _i , \beta _i )$. 
\end{center} 
 
Take $f\in \mathcal{D}'((U\cup V)\cap B(x,r))$ such that $f\mid _{U\cap B(x,r)} \in \widehat{\mathcal{F}}^k (U \cap B(x,r))$ and $f\mid _{ V\cap B(x,r)} \in \widehat{\mathcal{F}}^k ( V \cap B(x,r))$, clearly this implies that

\begin{center}
  $f\mid _{R(r,\gamma _i , \lambda _i )} \in \widehat{\mathcal{F}}^k (R(r,\gamma _i , \lambda _i ))$ and $f\mid _{R(r,\alpha _j , \beta _j )} \in \widehat{\mathcal{F}}^k (R(r,\alpha _j , \beta _j ))$ for all $i$ and $j$.
  \end{center}  
 
We want to prove that $f\mid _{(U\cup V)\cap B(x,r)} \in \widehat{\mathcal{F}}^k ((U \cup V)\cap B(x,r))$. By the local definition of $\widehat{\mathcal{F}}^k$, it is enough to prove that $f\mid _{C}\in \widehat{\mathcal{F}}^k(C)$ for every connected component $C$ of $(U \cup V)\cap B(x,r)$. So take $C$ a connected component of $(U \cup V)\cap B(x,r)$,  we can reorder the curves $\gamma _1,\lambda _1...,\gamma _m, \lambda _m ,\alpha _1,\beta _1...,\alpha _l, \beta _l $ to find definable curves $c_1 ,...,c_n$ such that

\begin{center}
 $C= \cup _{i=1} ^{n-2} R(r,c_{i} , c_{i+2})$, $f\mid _{R(r,c_{i} , c_{i+2})}\in \widehat{\mathcal{F}}^k (R(r,c_{i} , c_{i+2}))$, and $R(r,c_{i} , c_{i+2}) \cap R(r,c_{i+1} , c_{i+3})\neq \emptyset$ for any $i\in \{ 1,...,n-3  \}$.  
 \end{center}  

Using induction and the first step we deduce that $f\mid _{C}\in \widehat{\mathcal{F}}^k (C)$. 
 
\item \textbf{Case(B):} Let's be out of the assumption of \textbf{Case(A)} . Since we assumed that the germs $(\partial U, x)$ and $(\partial V, x)$ are not comparable, the only non-trivial case is when $U$,$V\in \{ C_2 ,..., C_6  \}$ and $U\cup V$ is like $C_1$. Let $L$ be a Lipschitz open subset in $U\cup V$. If $x\notin \overline{L}$, then $f\mid _{L} \in W^{k,2}(L)$ because for any $p\in \overline{L}$ there is a neighborhood $O_p$ of $p$ in $U$ or $V$ such that $f\mid _{O_p} \in W^{k,2}(O_p)$. Now, if $x\in \overline{L}$ then in this case near $x$, $L$ is like $C_2$ and covered by two open sets $U_L \in \{C_2 ,..., C_6   \}$ and $V_L \in \{C_2 ,..., C_6   \}$ such that $f\mid _{U_{L}}\in \widehat{\mathcal{F}}^k(U_{L})$ and  $f\mid _{V_{L}}\in \widehat{\mathcal{F}}^k(V_{L})$, and by the discussion of \textbf{Case(A)}, it follows that $f\mid _{L} \in \widehat{\mathcal{F}}^k(L)=W^{k,2}(L)$.

\end{itemize}

\end{proof}

\begin{rmk} \label{remark2.8} Take $k\in \mathbb{N}$. By analyzing each case, we can show that 

\begin{itemize}

\item[$(1)$] Let $U\in X_{\mathcal{A}}(\mathbb{R}^2)$, such that $U$ falls into  $C_1$ to $C_6$. Then, we have
\begin{center}
$\mathcal{F}^k (U)=\widehat{\mathcal{F}} ^k (U) $.
\end{center}
\item[$(2)$] If $W\in X_{\mathcal{A}}(\mathbb{R}^2)$ has only cuspidal singularities (singularities on the boundary of $W$ are Lipschitz or of type $C_3$), then

\begin{center}
$\mathcal{F} ^k (W)=W^{k,2}(W)$.
\end{center}

Consequently, if $U$ and $V$ belong to $X_{\mathcal{A}}(\mathbb{R}^2)$, such that $U$, $V$, $U \cap V$, and $U \cup V$ possess only cuspidal singularities, the sequence
\begin{center}
$0 \to W^{k,2}(U \cup V) \to W^{k,2}(U) \oplus W^{k,2}(V) \to W^{k,2}(U \cap V) \to 0$
\end{center}
is exact.

\item[$(3)$] For any $U \in X_{\mathcal{A}}(\mathbb{R}^2)$, the space $\mathcal{F}^k(U)$ naturally carries a Hilbert structure. Consider $\mathcal{L}=(L_1, L_2, ..., L_m)$ as an L-regular decomposition of $U$. Since each open L-regular set in $\mathbb{R}^2$ only contains cuspidal singularities, the following mapping
\[
\mathcal{N}_{\mathcal{L}}: \mathcal{F}^k(U) \longrightarrow \mathbb{R}, \quad f \mapsto \mathcal{N}_{\mathcal{L}}(f) = \sum_{\dim(L_i)=2} \| f_{|L_i} \|_{W^{k,2}(L_i)},
\]
defines a Hilbert structure on $\mathcal{F}^k(U)$ that is independent of $\mathcal{L}$. Furthermore, if $U$ exclusively has cuspidal singularities, this norm coincides with the Sobolev norm $\| \cdot \|_{W^{k,2}(U)}$.
\end{itemize}

\begin{proof}
Let's address each part of the proof step by step:

\begin{itemize}
\item \textbf{(1)} We proceed by considering different cases. The cases $C_1$ and $C_2$ follow straightforwardly from the fact that any $x\in \partial U$ (except for the center of the punctured disk) has a Lipschitz boundary in $U$. The case $C_6$ is a consequence of the additive property of $\mathcal{F}^k$ and the other cases. Therefore, we focus on proving $C_3$ and $C_4$ (where $C_5$ is analogous to $C_4$).

\begin{itemize}
\item \textbf{$C_3$ :} In this case, $U=R(r,\alpha, \beta)$ represents a cusp between boundary curves $\alpha$ and $\beta$. If $f\in \mathcal{F}^k (U)$, then for any $x\in \overline{U}$, there exists $r_x>0$ such that $f \mid _{U_{r_x} (x)} \in \widehat{\mathcal{F}} ^k  (U_{r_x} (x))=W^{k,2} (U_{r_x} (x))$. This holds because locally, on the boundary of $U$, the types are limited to $C_2$ and $C_3$. Thus, by using a partition of unity argument, we find that $f \in W^{k,2} (U)=\widehat{\mathcal{F}} ^ k (U)$. Similarly, if $f\in \widehat{\mathcal{F}} ^ k (U)= W^{k,2} (U)$, it is evident that $f\in \mathcal{F} ^ k (U)$ since $W^{k,2}$ is always a subspace of $\mathcal{F} ^ k$.

\item \textbf{$C_4$ :} In this case, we have two definable $C^1$-curves $\gamma _1 , \gamma _2 : [0,a[ \longrightarrow \mathbb{R}^2$ such that $\gamma _1 (0) = \gamma _2 (0)=p_0$, $\angle(\gamma _1 '(0),  \gamma _2 ' (0))=2\pi$, and

\begin{center}
$ U=R(r,\gamma _1 , \gamma _2)$.
\end{center}

Let $\gamma _3 , \gamma _4: [0,a[\longrightarrow \mathbb{R}^2$ such that $\gamma _3 (0)=\gamma _4 (0)=p_0 \in \mathbb{R}^2$, $\angle(\gamma _1 ' (0), \gamma _{3} ' (0))>0$, and $\angle(\gamma _4 ' (0), \gamma _{2} ' (0))>0$. Consequently,

\begin{center}
$\widehat{\mathcal{F}} ^k (U)=\{f\in \mathcal{D}'(U) \; : \; f_{\mid R(r,\gamma _1, \gamma _4)}\in W^{k,2} (R(r,\gamma _1, \gamma _4)) \; \text{and} \;  f_{\mid R(r,\gamma _3, \gamma _2)}\in W^{k,2} (R(r,\gamma _3, \gamma _2))    \}$.
\end{center}
For $f\in \widehat{\mathcal{F}} ^k (U)$ and  $x\in \partial U$, we can choose a sufficiently large $r$ so that $U_r (x)=U$ and $f \mid {U_r (x)} \in \widehat{\mathcal{F}} ^ k (U_r (x))$, implying $f\in \mathcal{F} ^ k (U)$. Conversely, consider $f\in \mathcal{F} ^ k (U)$. For the point $p_0$, we can find $r'>0$ such that $U_{r'} (p_0)=R(r',\gamma _1 , \gamma _2 )$ and $f \mid {U_{r'} (p_0)} \in \widehat{\mathcal{F}} ^ k (U_{r'} (p_0))$ due to the definition. This leads to

\begin{center}
 $\widehat{\mathcal{F}} ^ k (U_{r'} (p_0))=\{f\in \mathcal{D}'(U) \; : \; f_{\mid R(r',\gamma _1, \gamma _4)}\in W^{k,2} (R(r',\gamma _1, \gamma _4)) \; \text{and} \;  f_{\mid R(r',\gamma _3, \gamma _2)}\in W^{k,2} (R(r',\gamma _3, \gamma _2))    \} \quad (\star)$.
 \end{center} 
Considering that $U$ is Lipschitz near each point $x\in \partial U \setminus \{ p_0 \}$, it follows that $f$ is Sobolev near each of these points. Combining this with $(\star)$ shows that $f_{\mid R(r,\gamma _1, \gamma _4)}\in W^{k,2} (R(r,\gamma _1, \gamma _4)) \; \text{and} \;  f_{\mid R(r,\gamma _3, \gamma _2)}\in W^{k,2} (R(r,\gamma _3, \gamma _2))$, implying $f\in \widehat{\mathcal{F}} ^ k (U)$.
\end{itemize}

\item \textbf{(2)} When $W\in X_{\mathcal{A}}(\mathbb{R}^2)$ only possesses cuspidal singularities, consider any point $x\in \overline{W}$. There exists $r_x >0$ such that $W_{r_x} (x)$ is either Lipschitz or a standard cusp. Therefore, $\widehat{\mathcal{F}} ^ k (W_{r_x} (x)) = W^{k,2} (W_{r_x} (x))$. By using a partition of unity $(\phi _x)_{x\in W}$ for the covering $(W_{r_x}(x))_{x\in W}$, it's evident that

\begin{center}
$f=\sum _x \phi _x f\mid _{W_{r_x} (x)} \in W^{k,2} (W)$.
\end{center}
This establishes exactness on cuspidal domains.

\item \textbf{(3)} This result follows directly from the L-regular decomposition and $(2)$.
\end{itemize}

Thus, we have demonstrated each part of the remark.
\end{proof}

\end{rmk}

\textbf{Notation:} For $k\in \mathbb{N}$ and $U\in X_{\mathcal{A}} (\mathbb{R}^2)$ with only cuspidal singularities, we denote by $E_U$ a linear extension operator

\begin{center}
$E_U : W^{k,2} (U) \longrightarrow W^{k,2} (\mathbb{R}^2)$\\
$\; \; \; \; \; \; \; \; \; \; \; \; \; \; \; \; \; \; \; \; \; \; \; \; \; \; \; \; \; \; \; \; \; \; \; \; \; \; \; f\mapsto E_U (f) \; \text{with} \; (E_U (f)) _{\mid U} =f$. 
\end{center}
\section{Cohomology of the sheaf $\mathcal{F}^k$.}
For the cohomology computation, we need to introduce the concept of "good directions".

\textbf{Good directions:}
Consider a definable subset $A\subset \mathbb{R}^n$, with $\dim A <n$, and a unit vector $\lambda \in \mathbb{S}^{n-1}$. We say that $\lambda$ is a \emph{good direction} for $A$ if there exists $\varepsilon >0$ such that for all $x\in A^{reg}$, we have

\begin{center}
$d(\lambda , T_x A^{reg})>\varepsilon$.
\end{center}

Given $\lambda \in \mathbb{S}^{n-1}$, let $\pi _\lambda :\mathbb{R}^n \longrightarrow N_\lambda =<\lambda > ^{\perp}$ be the orthogonal projection, and let $x_\lambda$ denote the coordinate of $x$ along $<\lambda>$.

Consider definable sets $A\subset \mathbb{R}^n$ and $A' \subset N_\lambda$, along with a definable function $f:A' \longrightarrow \mathbb{R}$. We say that $A$ is the graph of the function $f$ with respect to $\lambda$ if 

\begin{center}
$A=\{x\in \mathbb{R}^n \; : \; \pi _\lambda (x)\in A' \; \text{and} \; x_\lambda =f(\pi _\lambda (x))     \}$.
\end{center}

It's important to note that $\lambda \in \mathbb{S}^{n-1}$ is a good direction for $A$ if and only if $A$ is a union of graphs of Lipschitz definable functions over certain subsets of $N_\lambda$. It's worth mentioning that the sphere $\mathbb{S}^n$ doesn't possess any good direction. To address this, we need to partition it into finite subsets, each of which has a distinct good direction. However, there exists a beautiful theorem by G. Valette \cite{V} which asserts that after applying a bi-Lipschitz deformation to the ambient space, a good direction can always be found:

\begin{thm}\label{Good direction} For any definable $X\subset \mathbb{R}^n$ with $\dim (X)<n$, there exists a definable bi-Lipschitz function $h:\mathbb{R}^n \longrightarrow \mathbb{R}^n$ such that $h(X)$ exhibits a  good direction $\lambda \in \mathbb{S}^{n-1}$.
\end{thm}

\begin{figure}[H]
  \centering
  \includegraphics[scale=0.4]{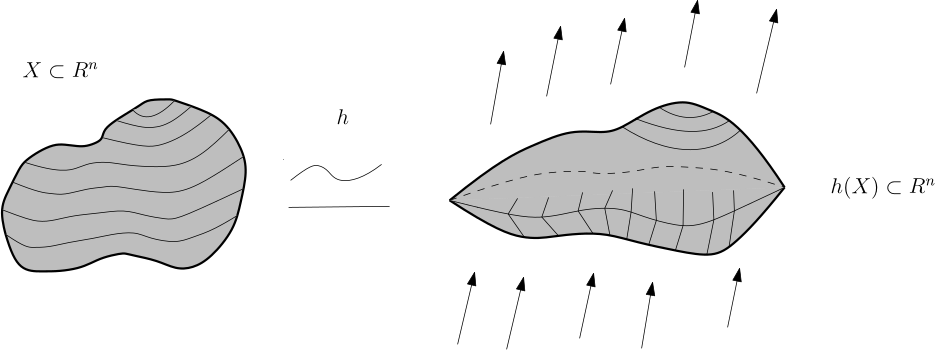}
  \caption{Example of a bi-Lipschitz transformation to get a good direction for a closed hypersurface in $\mathbb{R}^n$}   
\end{figure}

\begin{defi} Take $U\in X_{\mathcal{A}}(\mathbb{R}^2)$ and  $\mathcal{U}=(U_i)_{i\in I}$ a  cover of $U$ in the definable site $X_{\mathcal{A}}(\mathbb{R}^2)$. An \emph{adapted} cover of $\mathcal{U}$ is a definable cover $\mathcal{V}=\{ V_j \}_{j\in J}$ of $\mathbb{R}^2$ such that the following properties are satisfied:

\begin{itemize}

\item[(1)] $\mathcal{V}$ is compatible with $\mathcal{U}$, that is, each element in $\mathcal{U}$ is a finite union of elements in $\mathcal{V}$.

\item[(2)] Every finite intersection of elements in $\mathcal{V}$ is either empty or a connected domain with only cuspidal singularities, and the intersection of more than three elements is always empty.

\item[(3)] There exist $m\in \mathbb{N}$, $r>0$, and $(k_l,p_l)\in \mathbb{N}^2$ for each $l\in \{0,...,m+1 \} $ such that $\mathcal{V}=\{ V_j \}_{j\in J}$ can be rearranged as follows
	\begin{eqnarray*}
\mathcal{V} &=& \{ O_{l,p} \; : \; l\in \{0 ,1, ... , m+1 \} \; \text{and} \; p\in \{0,..., p_l \}   \}  \\
 &\cup & \{ \widehat{O}_{l,p} \; : \; l\in \{0 ,1, ... , m+1 \} \; \text{and} \; p\in \{0,..., p_l -1 \}   \} \\
 & \cup &  \{ V_{l,k} \; : \; l\in \{0 ,1, ... , m \} \; \text{and} \; k\in \{0,..., k_l +1 \}  \} \\
 & \cup & \{ B(a_{l,k},r) \; : \; a_{k,l} \in \mathbb{R}^2 , \; l\in \{0 ,1, ... , m \} \; \text{and} \; k\in \{0,..., k_l  \} \} 
\end{eqnarray*} 

\item[(4)] \begin{itemize} 
\item[$\bullet$] For each $l\in \{1, ... , m   \}$ and $p\in \{ 0 , ... , p_l \}$ there is a unique $(L(l,p) , R (l,p))\in \mathbb{N}^2$ such that the only possible non-Lipschitz singularities of $O_{l,p}$ and $\widehat{O}_{l,p}$ (only in the case of $p<p_l$) are at $a_{l-1 , L(l,p)}$ and $a_{l,R (l,p)}$.
\item[$\bullet$] For each $p\in \{ 0 , ... , p_0 \}$ there is a unique $ R (0,p)\in \mathbb{N}$ such that the only possible non-Lipschitz singularities of $O_{0,p}$ and $\widehat{O}_{0,p}$ (only in the case of $p<p_l$) is at $a_{l,R (0,p)}$.
\item[$\bullet$] For each $p\in \{ 0 , ... , p_{m+1} \}$ there is a unique $ L (m+1,p)\in \mathbb{N}$ such that the only possible non-Lipschitz singularities of $O_{m+1,p}$ and $\widehat{O}_{m+1,p}$ (only in the case of $p<p_l$) is at $a_{l,L (m+1,p)}$.  
\end{itemize}

\item[(5)] The only non empty intersections of two open sets in $\mathcal{V}$ are the open sets
 $O_{l,p} \cap \widehat{O} _{l,p}$, $\widehat{O} _{l,p} \cap O_{l,p+1}$, $O_{l,p}\cap V_{l,R(l,p)}$, $O_{l,p}\cap V_{l-1,L(l,p)}$, $B(a_{l,k},r) \cap V_{l,k} $, $B(a_{l,k},r) \cap V_{l,k+1} $, $B(a_{l-1,L(l,p)},r)\cap \widehat{O}_{l,p}$, $B(a_{l,R(l,p)},r) \cap \widehat{O}_{l,p}$, $B(a_{l-1,L(l,p)},r)\cap O_{l,p}$, and $B(a_{l,R(l,p)},r) \cap O_{l,p}$.     

\item[(6)] The only non empty intersections of three open sets in $\mathcal{V}$ are the open sets $O_{l,p}\cap V_{l,R(l,p)} \cap B(a_{l,R(l,p)},r)$, $O_{l,p}\cap V_{l-1,L(l,p)} \cap B(a_{l-1,L(l,p)},r)$, $\widehat{O}_{l,p}\cap V_{l,R(l,p)} \cap B(a_{l,R(l,p)},r)$, $\widehat{O}_{l,p}\cap V_{l-1,L(l,p)} \cap B(a_{l-1,L(l,p)},r)$, $O_{l,p} \cap \widehat{O}_{l,p} \cap B(a_{l,R(l,p)},r)$, $O_{l,p} \cap \widehat{O}_{l,p} \cap B(a_{l-1,L(l,p)},r)$, $O_{l,p+1} \cap \widehat{O}_{l,p} \cap B(a_{l,R(l,p+1)},r)$, and $O_{l,p+1} \cap \widehat{O}_{l,p} \cap B(a_{l-1,L(l,p+1)},r)$. 
\end{itemize}
\end{defi}
This definition is motivated by the construction in Figure 10 and explained in detail in the proof of Proposition~\ref{good refinement}. These covers will be essential in the computation of the cohomology of the sheaves $\mathcal{F}^k$ (see Theorem ~\ref{Cohomo}).

\textbf{\v{C}ech cohomology:} Recall that for a given sheaf $\mathcal{F}$ on a topological space $M$ and a covering $\mathcal{U}=(U_i)_{i\in I}$ with $I$ an ordered set, we have the \v{C}ech complex $\mathcal{C}^{\star } _{\mathcal{U}} (M, \mathcal{F})$  defined by
$$
\mathcal{C}^{0}_{\mathcal{U}} (M, \mathcal{F}) \overset{d_0}{\longrightarrow} \mathcal{C}^{1}_{\mathcal{U}} (M, \mathcal{F}) \overset{d_1}{\longrightarrow} \mathcal{C}^{2}_{\mathcal{U}} (M, \mathcal{F}) \longrightarrow \cdot \cdot \cdot 
$$
such that 
$$
\mathcal{C}^{m}_{\mathcal{U}} (M, \mathcal{F})= \bigoplus _{J=(i_0 < i_1 < ... <i_m)} \mathcal{F} (U_{J}), \; U_J = \cap _{j=0}^m
U_{i_j},
$$ 
and
$$
(d _{m} \alpha)_{U_J}:=(d _{m} \alpha)_{J=\{ i_0 < ... < i_{m+1}   \}}=\sum _{j=0 , ... ,m+1} (-1)^j (\alpha _{J\setminus i_j})_{\mid U_{J}}.
$$

Clearly, if $\mathcal{V}$ is a refinement of $\mathcal{U}$, then there is a canonical morphism $\mathcal{C}^ {\star } _{\mathcal{U}} (M, \mathcal{F} ) \longrightarrow \mathcal{C}^ {\star } _{\mathcal{V}} (M, \mathcal{F} )$. Thus, the \v{C}ech cohomology of degree $j$ of $M$ with respect to $\mathcal{F}$ is defined to be the colimit
$$
H^{j} (M,\mathcal{F})=\lim _{\mathcal{U}} H^{j} (\mathcal{C}^ {\star } _{\mathcal{U}} (M, \mathcal{F})).
$$

It is well known that this cohomology coincides with the cohomology of the sheaf $\mathcal{F}$ on paracompact spaces, and so on definable sets. We prove in the following proposition that any cover in the site $X _{\mathcal{A}} (\mathbb{R}^2)$ has an \emph{adapted} cover, and so we can use adapted covers to compute the cohomology of $\mathcal{F}^k$.
\begin{prop}\label{good refinement}
Take $U\in X_{\mathcal{A}} (\mathbb{R}^2)$ and $\mathcal{U}=(U_i)_{i\in I} \subset X_{\mathcal{A}} (\mathbb{R}^2)$ a cover of $U$. Then there is an adapted cover $\mathcal{V}$ of $\mathcal{U}$.
\begin{proof}
Take $\mathcal{U}=(U_i)_{i\in I}$ a definable cover of $U$. It is obvious that finding such a cover is sufficient after a bi-Lipschitz definable homeomorphism $h:\mathbb{R}^2 \longrightarrow \mathbb{R}^2$. Therefore, by Theorem~\ref{Good direction}, we can assume that $\bigcup _i \partial (U_i )$ is included in a finite union of graphs of definable Lipschitz functions $\xi _j : \mathbb{R} \longrightarrow \mathbb{R}$. We are going to construct an adapted cover $\mathcal{V}$ (see Figure 10). \\
Take 
$$
n=\max \{\sharp (\pi ^{-1}(x)\cap (\bigcup _j \Gamma _{\xi _j})) \; : \; x\in \mathbb{R}   \}<+\infty,
$$
where $\pi$ is the canonical projection:
$$\pi : \R ^2 \to \R$$
$$\; \; \; \; (x,y) \mapsto x$$
Consider $\mathcal{C}=\{]-\infty , a_0[ , \{ a_0 \} , ]a_0 , a_1[, ...  , ]a_{m-1} , a_{m}[ , \{ a_m \} , ]a_m , +\infty [ \} $ a cell decomposition of $\mathbb{R}$ compatible with the collection of sets
\begin{center}
$A_k=\{x\in \mathbb{R} \; : \; \sharp (\pi ^{-1}(x)\cap (\bigcup _j \Gamma _{\xi _j}))=k  \}$ for $k\in \{1,...,n \}$.
\end{center}
For $l\in \{ 0, ... , m \}$, we have $$\pi ^{-1}(a_l) \cap (\bigcup _j \Gamma _{\xi _j})=\{a_{l,0}  , a_{l,1}, ... ,a_{l, k_{l}} \}.$$ 
We denote $a_{-1}:=-\infty$ and $a_{m+1}:=+\infty$. For $l\in \{-1 ,0 , ..., m  \}$, there exist Lipschitz definable  functions
$$\phi _{l,0} < ... < \phi _{l,p_l} : ]a_l , a_{l+1}[\to \mathbb{R}$$  
such that $\pi ^{-1}\left( \pi (\bigcup _i \partial (U_i )) \cap ]a_l , a_{l+1}[ \right) \cap (\bigcup _j \Gamma _{\xi _j}) =\bigcup _p \Gamma _{\phi _{l,p}}$. For each $l\in \{-1,0,...,m \}$ and $p\in \{0,...,p_l \}$, there exist definable Lipschitz functions $\phi _{l, p} ^- < \phi _{l, p} ^+ :]a_l , a_{l+1}[\mapsto \mathbb{R}$, such that we have
$$\phi _{l, 0} ^- < \phi _{l, 0} < \phi _{l, 0} ^+ < \phi _{l, 1} ^- < \phi _{l, 1} < \phi _{l, 1} ^+ \dots <\phi _{l, p_l} ^+  ,$$
$$\lim _{t\rightarrow a_l} \phi _{l, p} ^- (t) = \lim _{t\rightarrow a_l} \phi _{l, p} ^+ (t) = \lim _{t\rightarrow a_l} \phi _{l, p} (t)  \; \text{for} \; l\in \{0, 1, \dots, m\},$$
and 
$$ \lim _{t\rightarrow a_{l+1}} \phi _{l, p} ^- (t) = \lim _{t\rightarrow a_{l+1}} \phi _{l, p} ^+ (t) = \lim _{t\rightarrow a_{l+1}} \phi _{l, p}  (t)  \; \text{for} \; l \in \{-1, 0, . . . , m-1\}.$$
Denote by $a_{l,-1}:=-\infty$ and $a_{l,k_l +1}:=+\infty$. For  each $l\in \{0, ... , m \}$ and $k \in \{-1,...,  k_l   \}$,  there exist Lipschitz functions (with respect to the direction $e = (1,0)$) 
$$\varphi _{l,k} ^-  < a_l <  \varphi _{l,k} ^ + :]\pi _e (a_{l,k}) , \pi _e(a_{l,k+1})[\to \mathbb{R}$$ such that the graphs of these functions do not intersect the graphs of the functions $\phi _{l, p} ^s $ (for any $l$, $p$, and $s\in  \{0,-,+ \}$ with $\phi _{l,p} ^0 := \phi _{l,p}$), and
$$\lim _{t\rightarrow \pi _e (a_{l,k})} \varphi _{l, k} ^s (t)= a_l= \lim _{t\rightarrow \pi _e (a_{l,k+1})} \varphi _{l, k} ^s (t).$$
For each $(l,k)$ such that $a_{l,k}\in U$, there exists $r_{l,k}>0$ such that $B(a_{l,k} , r_{l,k})\subset U_i$ for all $U_i$ that contain $a_{l,k}$.  Choose $r<\min_{l,k} (r_{l,k})$ such that $\partial B(a_{l,k},r)$ is transverse to all the graphs of the functions $\phi _{l,p} ^s$ and  $\varphi _{l,k} ^s$ (here also $\varphi _{l,k} ^0 := a_{l,k}$), with
$$\overline{B}(a_{l',k'},r) \cap \overline{B}(a_{l,k},r)=\emptyset \; \text{if} \; (l,k)\neq (l',k').$$
Consider the collection of open definable sets (following the terminology of Definition 7.2)
$$\mathcal{V} = \{O_{l,p}, \widehat{O}_{l,p}, V_{l,k}, B(a_{l,k},r) \;    \} _{l,k,p},$$
where
$$O_{l,p}= \Gamma (\phi _{l,p}  , \phi _{l,p+1} ), \widehat{O}_{l,p} = \Gamma (\phi _{l,p} ^- ,\phi _{l,p} ^+),\; \text{and} \; V_{l,k}=\Gamma (\varphi _{l,k} ^- ,\varphi _{l,k} ^+) .$$
Clearly, the collection $\mathcal{V}$ is an adapted cover of $\mathcal{U}$.
\end{proof}
\end{prop}

Figure 10 illustrates an example of an adapted cover $\mathcal{V}$ following the notation used in the proof of Proposition~\ref{good refinement}.

\begin{figure}[H]
  \centering
  \includegraphics[scale=0.45]{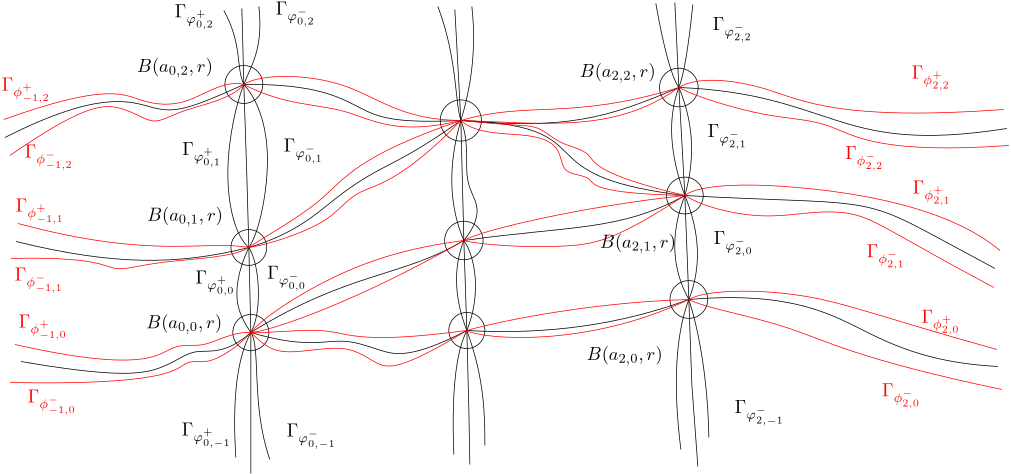}
  \caption{The cover $\mathcal{V}$.}   
\end{figure}

\begin{rmk}\label{description of the good cover}
\begin{itemize}
\item[(1)] \label{example of H1} The sheaf $\mathcal{F}^k : X_{\mathcal{A}}(\mathbb{R}^2) \longrightarrow \mathbb{C}-$vector spaces is not acyclic for $k>1$. In this case, we have an inclusion $W^{k,2} (\mathbb{R}^2) \subset C^0 (\mathbb{R}^2)$. Consider the punctured disk $W=B(0,1) \setminus \{ 0 \} = U \cup V$, with

\begin{center}
$U=\{(x,y)\in W: \; y > x \; or \; y<-x   \}$ and $V=\{(x,y)\in W: \; y> -x \; or \; y< x   \}$.
\end{center}
And $U\cap V= O_1 \cup O_2$ with $\overline{O_1}\cap \overline{O_2}=\{ 0 \}$, where

\begin{center}
$O_1=\{(x,y)\in W: \; y> \mid x \mid \}$ and $O_2 = \{(x,y)\in W: \; y< - \mid x \mid \}$.
\end{center}

If $H^1 (W, \mathcal{F} ^k )=0$, then by the Mayer-Vietoris long exact sequence, the sequence
\begin{center}
$0\mapsto \mathcal{F} ^k (W) \mapsto W^{k,2} (U) \oplus W^{k,2}(V) \mapsto W^{k,2}(O_1) \oplus W^{k,2} (O_2) \mapsto 0$
\end{center}
 is exact. However, this is not possible because for a constant couple $(f= 1, g = 0)\in W^{k,2}(O_1) \oplus W^{k,2} (O_2)$ there are no continuous functions $(u,v)\in W^{k,2} (U) \oplus W^{k,2}(V)$ such that $(u-v)\mid _{O_1} =1$ and  $(u-v)\mid _{O_2} =0$.
Hence $H^1 (W, \mathcal{F} ^k )\neq 0$.

\item[(2)] In Theorem~\ref{Cohomo} we will compute the cohomology of $\mathcal{F}^k$. The proof of Theorem~\ref{Cohomo} is based on the following observations: from the construction of \emph{adapted} covers, we can deduce that $H^j(\cdot,\mathcal{F}^k)=0$ for $j\geqslant 2$ (because the intersection of more than three elements in the adapted cover is always empty). For the first cohomology groups of $\mathcal{F}^k$, the only obstruction for $H^1 (U, \mathcal{F} ^k)$ to  vanish is the existence of punctured disks in $U$. If we take $U$ with no punctured disk singularity, then locally gluing cocycles from $\mathcal{C}_{\mathcal{U}}^1 (U, \mathcal{F}^k)$ to cochains in $\mathcal{C}_{\mathcal{U}}^0 (U, \mathcal{F}^k)$ is summarized in the following simple example: take $x_0\in \overline{U}$ and $\gamma _0 , \cdots , \gamma _4=\gamma _0 :[0,a[\to \overline{U}$ (see Figure 11) with $\gamma _0 (0)= \cdots =\gamma _4 (0)=x_0$ and $\gamma _i ^- < \gamma _i ^+ :[0,a[\to \mathbb{R}^2  $ (see Figure 12). Then locally, we choose two situations (in fact they are the only situations that will show up locally in the proof of Theorem~\ref{Cohomo}): 
\begin{figure}[H]
  \centering
  \includegraphics[scale=0.4]{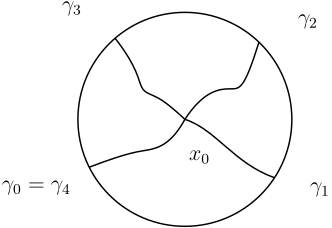}
  \label{casep}
  \caption{The curves $\gamma _i$ around $x_0$.}   
\end{figure}
\begin{figure}[H]
  \centering
  \includegraphics[scale=0.4]{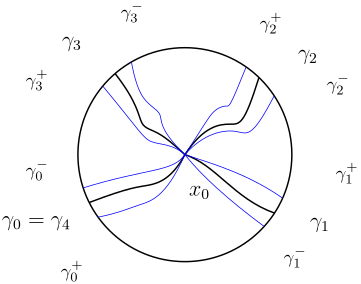}
  \label{casep2}
  \caption{The curves $\gamma _i ^-$ and $\gamma _i ^+$.}   
\end{figure}

\begin{figure}[H]
  \centering
  \includegraphics[scale=0.5]{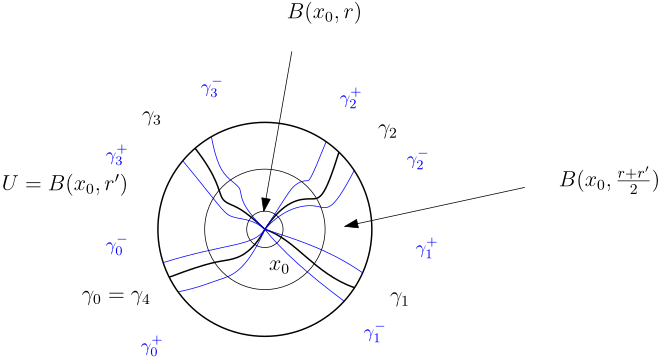}
  \label{casep2}
  \caption{The covering of $U$ in Situation 1.}   
\end{figure}

\begin{itemize}
\item \textbf{Situation 1:} Assume that $x_0 \in U$. In this situation, for some $0<r<r'$, we assume that (see Figure 13)
$$U=\bigcup _i R(r',\gamma _i , \gamma _{i+1}) \bigcup _i R(r',\gamma _i ^- , \gamma _i ^+) \bigcup B(x_0,r).$$ For each $i$, we take functions $f_{i,+}\in W^{k,2}(R(r', \gamma _i, \gamma _i^+))$, $f_{-,i}\in W^{k,2}(R(r', \gamma _i ^- , \gamma _i))$, $g_i \in W^{k,2}(R(r, \gamma _i , \gamma _{i+1}))$, and $h_i \in W^{k,2}(R(r,\gamma _i ^- , \gamma _i ^+ ))$ such that
\begin{center}
$  (f_{i,+}) _{\mid R(r, \gamma _i , \gamma _i^+ )} =  (g_{i}) _{\mid R(r, \gamma _i , \gamma _i^+ )} $,\\
$  (f_{-,i}) _{\mid R(r, \gamma _i ^- , \gamma _i )} = (g_{i-1}) _{\mid R(r, \gamma _i ^- , \gamma _i )}   $,\\
$(h_{i}) _{\mid R(r, \gamma _i , \gamma _i^+ )} =  (g_{i}) _{\mid R(r, \gamma _i , \gamma _i^+ )} $, and \\
$  (h_{i}) _{\mid R(r, \gamma _i ^- , \gamma _i )} =  (g_{i-1}) _{\mid R(r, \gamma _i ^- , \gamma _i )}  $.
\end{center}
       
We want to glue these functions to functions in $W^{k,2} (R(r',\gamma _i , \gamma _{i+1}))$, $W^{k,2} (R(r',\gamma _i ^- , \gamma _{i} ^+))$, and $W^{k,2}(B(x_0,r))$. Take $(\phi ' , \phi  ) $ a partition of unity associated to the covering $\{ C=B(x_0 ,r')\setminus B(x_0 , r) , B(\frac{r+r'}{2},x_0) \}$ (see Figure 13). Define  $u\in W^ {k,2}(B(x_0 , r))$ by taking just the values of $g_i's$ and $h_i's$. On each $R(r',\gamma _i ^- , \gamma _i ^+)$, we choose the zero function. We take smooth compactly supported functions $F_i:\mathbb{R}^2 \to [0,1] $ such that $F_i=1$ on a neighborhood of $R(r',\gamma _i ^- , \gamma _i ^+) \cap C$ and $F_i=0$ on the other sets of type $R(r',\gamma _j ^- , \gamma _j ^+) \cap C$. So, in each $W^{k,2}(R(r',\gamma _i , \gamma _{i+1}))$ we define $v_i$ by 
$$v_i:=\left( \phi ' (F_iE_{R(r',\gamma _i , \gamma _i ^+)}(f_{i,+})+F_{i+1}E_{R(r',\gamma _{i+1}^- , \gamma _{i+1})}(f_{i+1,-})) + \phi (E_{B(x_0 , r)}(u)) \right) _{\mid R(r',\gamma _i , \gamma _{i+1})} .$$  
Then, clearly, the functions $u$, $0$ and $v_i$ glue the functions $f_{i,+}$, $f_{-,i}$, $g_i$ and $h_i$. 

\begin{figure}[H]
  \centering
  \includegraphics[scale=0.4]{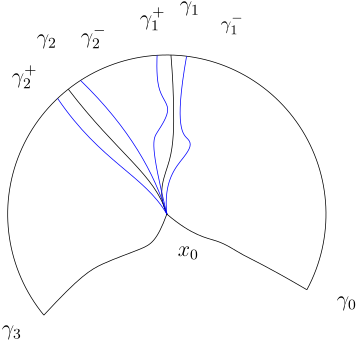}
  \label{casep2}
  \caption{The covering of $U$ in Situation 2.}   
\end{figure}

\item \textbf{Situation 2:} Assume that $x_0 \notin U$. In this situation (with the same notation as in the first case), we assume that (see Figure 14)
$$U=R(r', \gamma _0 , \gamma _1) \bigcup R(r', \gamma _1 ^- , \gamma _1 ^+) \bigcup R(r', \gamma _1 , \gamma _2) \bigcup R(r', \gamma _2 ^-, \gamma _2 ^+) \bigcup R(r', \gamma _2 , \gamma _3),$$
with a given functions $f_{i,+}\in W^{k,2}(R(r', \gamma _i , \gamma _i^+))$ and $f_{-,i}\in W^{k,2}(R(r', \gamma _i ^- , \gamma _i))$. To glue these functions, it is  enough to take the functions  
\begin{center}
$v_0:=0 \in W^{k,2}(R(r', \gamma _0 , \gamma _1)),$\\
 $u_1 :=\left( E_{R(r', \gamma _1 ^- , \gamma _1)}(f_{- ,1} ) \right) _{\mid R(r', \gamma _1 ^- , \gamma _1 ^+ )} \in W^{k,2} (R(r', \gamma _1 ^- , \gamma _1 ^+ )),$ \\
 
 $v_1 := \left( E_{R(r', \gamma _1 ^- , \gamma _1 ^+)} (f_{1,+}) + E_{R(r', \gamma _1 ^- , \gamma _1 ^+ )} (u_1) \right) _{\mid R(r', \gamma _1 , \gamma _2)} \in W^{k,2}(R(r', \gamma _1 , \gamma _2)),$\\
 $u_2:=\left( E_{R(r', \gamma _2 ^- , \gamma _2 )} (f_{-,2}) + E_{R(r', \gamma _1 , \gamma _2)} (v_1)   \right) _{\mid R(r', \gamma _2 ^- , \gamma _2 ^+)} \in W^{k,2} (R(r', \gamma _2 ^- , \gamma _2 ^+)),$\\
 $v_2:=\left( E_{R(r', \gamma _2  , \gamma _2 ^+ )} (f_{2,+}) + E_{R(r', \gamma _2 ^- , \gamma _2 ^+)} (u_2)  \right) _{\mid R(r', \gamma _2 , \gamma _3)} \in W^{k,2} (R(r', \gamma _2 , \gamma _3)).$
\end{center}

\end{itemize}

\end{itemize}
 
\end{rmk}
For the sake of notation, in the following theorem we will use $\mathcal{F}$ instead of $\mathcal{F}^k$ and $W^2$ instead of $W^{k,2}$.
\begin{thm}\label{Cohomo}
Take $U\in X_{\mathcal{A}}(\mathbb{R}^2)$. Then for any $j>1$ we have
\begin{center}
$H^j (U, \mathcal{F})=0$. 
\end{center}

And if $U$ has no singularities of type $C_1$, then for any $j \in \mathbb{N}$
\begin{center}
$H^j (U, \mathcal{F})=\left\{
                        \begin{array}{ll}
                  \mathcal{F}(U) \; \; \; \; \; \;  \; \; \; \; \text{if} \;  j=0 \\
                        \{0\} \; \; \; \; \; \; \; \; \; \; \; \; \; \; \; \text{if} \; j\geqslant 1. \\
                        \end{array}
                        \right.$
\end{center}

\begin{proof}
By the definition of the \v{C}ech cohomology, it is enough to compute the \v{C}ech cohomology on an adapted cover. So take $\mathcal{V}$ an adapted cover of $\{ U \}$ as given by Proposition~\ref{good refinement} and take $\mathcal{W}$ the cover of $U$ defined by
$$\mathcal{W}=\{O\in \mathcal{V} \; : \; O\subset U  \}.$$
Then we have the \v{C}ech complex
$$
\mathcal{C}^{0}_{\mathcal{W}} (U, \mathcal{F}) \overset{d_0}{\to} \mathcal{C}^{1}_{\mathcal{W}} (U, \mathcal{F}) \overset{d_1}{\to} \mathcal{C}^{2}_{\mathcal{W}} (U, \mathcal{F})\to 0.  
$$
For $j > 2$, we have $\mathcal{C}^j_{\mathcal{W}}(U, \mathcal{F}) = 0$, because the intersection of four elements in $\mathcal{W}$ is always empty. Take $\omega \in \mathcal{C}^2_{\mathcal{W}}(U, \mathcal{F})$. Let $\mathcal{W}_2$ be the set of all nonempty intersections of three elements of $\mathcal{W}$, and let $\mathcal{W}_1$ be the set of all nonempty intersections of two elements of $\mathcal{W}$. So we can write $\omega$ as follows
\[
\omega = \sum_{W \in \mathcal{W}_2} \omega(W),
\]
where for $O \in \mathcal{W}_2$ we define
\[
(\omega(W))_O = \left\{
\begin{array}{ll}
(\omega)_W & \text{if } O = W, \\
0 & \text{if } O \neq W.
\end{array}
\right.\]
To show that $\omega = 0$ in $H^2(U, \mathcal{F})$, it is enough to find for each $W \in \mathcal{W}_2$ an element $\alpha(W) \in \mathcal{C}^1_{\mathcal{W}}(U, \mathcal{F})$ such that $d(\alpha(W)) = \omega(W)$. For each $a_{k,l} \in \mathbb{R}^2$, we take a smooth function $F_{k,l} \in C_c^{\infty}(\mathbb{R}^2)$ such that $F_{k,l} = 1$ on $B(a_{k,l}, r)$ and $F_{k,l} = 0$ on each other $B(a_{k',l'}, r)$.  
Take $W \in \mathcal{W}_2$. Then $W = B(a_{l,k}, r) \cap Y$, where $Y$ is one of the cases in (5) of Definition 7.2. For any $O \in \mathcal{W}_1$ we define
\[
(\alpha(W))_O = \left\{
\begin{array}{ll}
\left(\sum_{k,l} F_{k,l} E_W((\omega)_W)\right)_{\mid Y} & \text{if } O = Y, \\
0 & \text{if } O \neq Y.
\end{array}
\right.
\]

So clearly we have $d(\alpha(W)) = \omega(W)$, and so $H^2(U, \mathcal{F}) = 0$.\\
Now assume that $U$ has no punctured disk singularity, and let's show that $H^1(U,\mathcal{F})=0$. Take $\alpha \in \mathcal{C}_{\mathcal{W}}^1(U,\mathcal{F})$ such that $d(\alpha)=0$, so we need to find $u\in \mathcal{C}_{\mathcal{W}}^0(U,\mathcal{F})$ such that $d(u)=\alpha$. For $O\in \mathcal{W}$, we define $u_O \in \mathcal{C}_{\mathcal{W}}^0(U,\mathcal{F})$ by induction on $l$ and $p$ (see (3) of Definition 7.2):
\begin{itemize}

\item \textbf{$O=O_{0,0}$:} In this case we define $u _O=0\in W^{s,2}(O).$

\item \textbf{$O=\widehat{O}_{0,p} $:} Assuming that we have constructed $u _{ O_{0,p} }$, we define $u _O \in W^{s,2}(O)$ by
$$u _O=\left( E_{O_{0,p}}(u _{O_{0,p}})+E_{O_{0,p} \cap \widehat{O}_{0,p}} (\alpha _{O_{0,p} \cap \widehat{O}_{0,p}}) \right) _{\mid O}. $$
\item \textbf{$O= O_{0,p+1}$:} Assuming that we have constructed $u _{\widehat{O}_{0,p}}$ , we define $u_O \in W^{s,2}(O)$ by
 $$u _O=\left( E_{\widehat{O}_{0,p}}(u _{\widehat{O}_{0,p}})+E_{\widehat{O}_{0,p} \cap O_{0,p+1}} (\alpha _{\widehat{O}_{0,p} \cap O_{0,p+1}}) \right) _{\mid O}. $$
 This was induction on $p$ with fixed $l=0$. Now assume that for $l$ fixed we have constructed $u _{O_{l,p}} $ and $u _{\widehat{O}_{l,p}}$ for each $p$. If $O=V_{l,k} \in \mathcal{W}$, then by (4) of Definition 7.2 there is a unique $p$ such that 
$$O_{l,p} \cap 	V_{l,k}\neq \emptyset.$$
In this case we define $u_O$ by 
$$u _O = \left( E_{O_{l,p}} (u _{O_{l,p}})+ E_{O_{l,p} \cap 	V_{l,k}} (\alpha _{O_{l,p} \cap 	V_{l,k}}) \right) _{\mid O} .$$
To finish, we need to construct $u$ on each $O=O_{l+1,p}$ and $O=\widehat{O}_{l+1,p}$ for each $p$. We discuss the following cases:

\item \textbf{$O=O_{l+1,0}$:}
  Assume that there is a unique $k$ such that $O\cap V_{l,k} \neq \emptyset$ (if not we define $u _O$ to be $0$), so we define $u _O$ by
 $$u _O=\left( E_{V_{l,k}}(u _{V_{l,k}})+E_{O\cap V_{l,k}} (\alpha _{O\cap V_{l,k}}) \right) _{\mid O}. $$ 

\item \textbf{$O=\widehat{O}_{l+1,p}$:} Assume that we have constructed $u_{O_{l+1,p}}$. We define $u _O$ by
$$u _O:=\left( E_{O_{l+1,p}} (u_{O_{l+1,p}}) +E_{\widehat{O}_{l+1,p} \cap O_{l+1,p}} (\alpha _{\widehat{O}_{l+1,p} \cap O_{l+1,p}})  \right) _{\mid O}.$$ 

\item \textbf{$O=O_{l+1,p+1}$:} We break it into two cases:
\begin{itemize}
\item \textbf{Case(1):} For any $k$ we have $V_{l+1,k} \cap O=\emptyset$. We define $u_O$ by
$$u _O := \left( E_{\widehat{O}_{l+1,p}} (u_{\widehat{O}_{l+1,p}}) + E_{\widehat{O}_{l+1,p} \cap O_{l+1,p+1}} (\alpha _{\widehat{O}_{l+1,p} \cap O_{l+1,p+1}}) \right) _{\mid O}.$$ 
\item \textbf{Case(2):} There exists $k$ such that
$$V_{l+1,k} \cap O \neq \emptyset .$$ 
In this case, $B(a_{l+1,k},r)\in \mathcal{W}$  (because otherwise $a_{l+1,k}$ will be a punctured disk singularity for $U$), and we choose $u_{B(a_{l+1,k},r)}$ to take the values of $\alpha$. Take $r'>r$ such that $B(a_{l+1,k},r')\cap B(a_{l+1,k+1},r')=\emptyset$ and $(f,g)$ a partition of unity associated to the cover $(B(a_{l+1,k},r'), \mathbb{R}^2 \setminus  B(a_{l+1,k},r)) $. We also take $h$ and $h' \in C^\infty (\mathbb{R}^2)$ such that 
\begin{center}
$h _{\mid V_{l+1,k} \cap \mathbb{R}^2 \setminus  B(a_{l+1,k},r))}=0$, $h _{\mid \widehat{O}_{l+1,p} \cap \mathbb{R}^2 \setminus  B(a_{l+1,k},r))}=1,$\\
$h' _{\mid V_{l+1,k} \cap \mathbb{R}^2 \setminus  B(a_{l+1,k},r))}=1$, and $h' _{\mid \widehat{O}_{l+1,p} \cap \mathbb{R}^2 \setminus  B(a_{l+1,k},r))}=0$.
\end{center}
So in this case we define $u_O$ by
\begin{center}
$u_O:=h \left( f E_{B(a_{l+1,k},r)}(u_{B(a_{l+1,k},r)}) + g E_{\widehat{O}_{l+1,p}} (u_{\widehat{O}_{l+1,p}})  \right)_{\mid O} +$\\ $h'\left( f E_{B(a_{l+1,k},r)}(u_{B(a_{l+1,k},r)}) + g E_{V_{l+1,k}} (u_{V_{l+1,k}})  \right)_{\mid O} .$
\end{center}
And, for any $O$ such that $a_{l+1,k}\in \overline{O}$, we need to modify the definition of $u_O$ by
$$u_O:= \left(f E_{B(a_{l+1,k},r)}(u_{B(a_{l+1,k},r)})+g E_{O}(u'_O)  \right)_{\mid O}, $$
where $u' _O$  is the old definition given in the previous stages of the induction.
\end{itemize}

\end{itemize}

Finally, from the construction of $u$, we have $d(u)=\alpha$.

\end{proof}
\end{thm}

\section{ $(W^{1,2} , W^{0,2})$-double extension is a sufficient condition for the sheafification of $W^{s,2}$.} 

\vspace{0.3cm}

In this section, we provide a categorical proof of Lemma~\ref{Proposition2.2}, and we discuss the case where $U\cap V$ is not Lipschitz. The only assumption we require here is that $U$, $V$, and $U\cup V$ are Lipschitz. We use the fact that the sequences 

\begin{center}
$0\to W^{0,2} (U\cup V)\to W^{0,2} (U)\oplus W^{0,2} (V) \to W^{0,2} (U\cap V) \to 0$
\end{center}
and
\begin{center}
$0\to W^{1,2} (U\cup V)\to W^{1,2} (U)\oplus W^{1,2} (V) \to W^{1,2} (U\cap V) \to 0$
\end{center}
are exact (see Lemma~\ref{Proposition2.2} and Remark~\ref{rem_s_in_N}).\\

We assume that we have the following double extension:

\textbf{Assumption:} There exists a linear continuous extension operator
\begin{center}
 $\mathcal{T}:W^{0,2}(U\cap V)\longrightarrow W^{0,2} (\mathbb{R}^n)$,
  \end{center}
such that $\mathcal{T}$ induces a linear continuous extension from $W^{1,2}(U\cap V)$ to $W^{1,2}(\mathbb{R}^n)$. \\

\begin{rmk}
Note that this assumption holds if $U\cap V$ is Lipschitz, due to the Stein extension theorem.
\end{rmk}

\vspace{0.3cm}

Note that here $W^{0,2}=L^2$, and we need only Sobolev spaces with regularity $s\in ]0,1[$. This is because, as recalled from Section 3, if $k\in \mathbb{N}$ and $s\in ]k,k+1[$, then $s-k \in ]0,1[$ and for $U\subset \R ^n$:
$$W^{s,2}_\star(U) = \{ f\in L^2(U) \; : \; \partial^\alpha f \in W^{s-k,2}_\star(U) \text{ for all } |\alpha| \leq k \}.$$
We will pass to our exact sequence for $s\in ]0,1[$ by a linear combination of the last two, which leads us to expect it to be exact. To achieve this, we will use the notion of \textbf{exact category} (see \cite{E}). An exact category is not abelian but has a structure that enables us to perform homological algebra.\\

Let $\mathcal{C}$ be an additive category. A pair of composable morphisms
\begin{center}
\begin{tikzcd}
X \arrow[r, "f"] & Y \arrow[r, "g"] & Z
\end{tikzcd}
\end{center}
is said to be a \textbf{KC-pair} (Kernel-Cokernel pair) if $f$ is the kernel of $g$ and $g$ is the Cokernel of $f$. Fix $\mathcal{E}$ as a class of KC-pairs. An \textbf{admissible monomorphism} (with respect to $\mathcal{E}$) is a morphism $f$ such that there is a morphism $g$ with $(f,g) \in \mathcal{E}$. \textbf{Admissible epimorphisms} are defined dually.\\

\begin{defi}
An \textbf{exact structure} is a pair $(\mathcal{C}, \mathcal{E})$ where $\mathcal{C}$ is an additive category and $\mathcal{E}$ is a class of KC-pairs, closed under isomorphisms, and satisfying the following proprieties:
\begin{description}
\item[$(E_0)$] For any $X\in Obj(\mathcal{C})$, $Id_X$ is an admissible monomorphism.
\item[$(E_0)^c$] The dual statement of $(E_0)$.
\item[$(E_1)$] The composition of admissible monomorphisms is an admissible monomorphism.
\item[$(E_1)^c$] The dual statement of $(E_1)$.
\item[$(E_2)$] If $f:X\to Y$ is an admissible monomorphism and $t:X\to T$ a morphism, then the pushout 
\begin{center}
\begin{tikzcd}
X \arrow[r, "f"] \arrow[d, "t"'] & Y \arrow[d, "s_Y"] \\
T \arrow[r, "s_T"']              & S                 
\end{tikzcd}
\end{center}
exists and $s_T$ is an admissible monomorphism.
\item[$(E_2)^c$] The dual statement of $(E_2)$.
\end{description}
\end{defi}
If $(\mathcal{C}, \mathcal{E})$ is an exact structure, a morphism $f:X\longrightarrow Y$ is said to be \textbf{$\mathcal{E}$-strict} if it can be decomposed into
\begin{center}
\begin{tikzcd}
X \arrow[rr, "f"] \arrow[rd, "e"'] & {} \arrow[loop, distance=2em, in=305, out=235] & Y \\
                                   & Z \arrow[ru, "m"']                             &  
\end{tikzcd}
\end{center}
 where $e:X\longrightarrow Z$ is an admissible epimorphism (with respect to $\mathcal{E}$), and  $m:Z\longrightarrow Y$ is an admissible monomorphism (with respect to $\mathcal{E}$).\\

Now fix $\mathcal{C}$ an additive category. It is well known (see \cite{E}) that the following class of KC-pairs
\begin{center}
$\mathcal{E}_{0} =\{ (f,g) \; : \;   \begin{tikzcd}
X \arrow[r, "f"] & Y \arrow[r, "g"] & Z
\end{tikzcd} \; \text{split}    \}$
\end{center}

is an exact structure on $\mathcal{C}$ (it is the smallest one on $\mathcal{C}$).\\

\begin{defi} Let $(\mathcal{C}, \mathcal{E})$ be an exact structure, $\mathcal{A}$ an \textbf{abelian} category, and $F:\mathcal{C}\longrightarrow \mathcal{A}$ an \textbf{additive functor}. $F$ is said to be \textbf{injective} if for any pair $\begin{tikzcd}
X \arrow[r, "f"] & Y \arrow[r, "g"] & Z
\end{tikzcd}$ in $\mathcal{E}$, the sequence 
\begin{center}
  \begin{tikzcd}
0 \arrow[r] & F(X) \arrow[r, "F(f)"] & F(Y) \arrow[r, "F(g)"] & F(Z)
\end{tikzcd}
  \end{center}  
is exact in $\mathcal{A}$.
\end{defi}
The following result is well known in the theory of exact categories:

\begin{prop}\label{Inj-functors}
$F$ is injective if and only if it preserves the kernel of every $\mathcal{E}$-strict morphism.
\begin{proof}
See \cite{E}.
\end{proof}
\end{prop}

We will construct a category $\mathcal{C}$ to serve our case, and the category $\mathcal{A}$ will be just the category of $\mathbb{C}$-vector spaces. Let's recall the concept of interpolation:

\begin{defi} A good pair of Banach spaces (or \textbf{GB-pair}) is a pair $(X,Y)$ of Banach spaces such that $X\subset Y$ with continuous inclusion, that is, there is $C>0$ such that for any $x\in X$ we have
$$\| x \| _Y \leqslant C \| x \| _X.$$
\end{defi}

We recall the interpolation $K$-method. So fix $(X,Y)$ a GB-pair  and $t>0$, and define the $K$-norm on $Y$ by 
$$u\mapsto K(t,u)=\inf\{\| x \| _{X}+
t\| y \| _{Y} \; : \; u=x+y, \; x\in X, \; y\in Y     \}.$$
For $s\in ]0,1[$, we define the interpolation space $[X,Y]_{s}$ by 
$$[X,Y]_{s}=\{u\in Y \; : \; \int _{0} ^{+\infty}\left( t^{-s}K(t,u) \right) ^{2}\frac{dt}{t}<+\infty   \}.$$
It is a Banach space with the norm 
$$\| u \| _{[X,Y]_{s}}=\left( \int _{0} ^{+\infty}\left( t^{-s}K(t,u) \right) ^{2}\frac{dt}{t}\right) ^{\frac{1}{2}}.$$
Recall the following theorem on interpolation spaces:
\begin{thm}\label{Interpolation}
Let $(X,Y)$ and $(X',Y')$ be two GB-pairs and 
\begin{center}
$L:Y\longrightarrow Y'$
\end{center}
a continuous linear map such that $L$ induces a continuous linear map from $X$ to $X'$. Then, for any $s\in ]0,1[$, $L$ induces a linear continuous map from $[X,Y]_s$ to $[X',Y']_s$.
 \begin{proof}
 See \cite{J}.
\end{proof}    

\end{thm}

Let $\mathcal{A}$ be the category of $\mathbb{C}$-vector spaces and  $\mathcal{C}$ be the category where the object are $GB$-pairs. For  $((X,Y),(X',Y'))\in (Obj(\mathcal{C}))^2$, we define the morphisms as: 
\begin{center}
$Hom _{\mathcal{C}}((X,Y),(X',Y'))=\{ L\in \mathcal{L}(Y,Y') \; : \; L\big|_X
 \in \mathcal{L}(X,X')  \; \}$.
\end{center}

Clearly, $\mathcal{C}$ is an additive category. We consider the exact structure $\mathcal{E}_0$ on $\mathcal{C}$ of splitting KC-pairs. For any $s\in ]0,1[$, we define the functor $F_{s}: \mathcal{C}\longrightarrow \mathcal{A}$ as follows
\begin{center}
 $F_{s} ((X,Y))=[X,Y]_s$ and for $f\in Hom _{\mathcal{C}}((X,Y),(X',Y'))$\\
 $F_s (f)=f\mid _{[X,Y]_s}$.
 \end{center} 
By Theorem~\ref{Interpolation} , $F_s$ is well defined additive functor. 
\begin{lem}\label{some} For $(X,Y),(X',Y')\in Obj(\mathcal{C})$ and for $s\in [0,1]$, there is a natural isomorphism 

\begin{center}
$[X\oplus X' , Y\oplus Y']_s \simeq[X,Y]_s \oplus [X',Y']_s$.
\end{center}
\begin{proof}
Take the projections 
\begin{center}
$P:Y\oplus Y' \longrightarrow Y$, and 
\end{center}

\begin{center}
$P':Y\oplus Y' \longrightarrow Y'$.
\end{center}

Since $P\mid _{X\oplus X'}\in \mathcal{L}(X\oplus X' ,X)$ and $P'\mid _{X\oplus X'}\in \mathcal{L}(X\oplus X', X')$, by Theorem~\ref{Interpolation} this induces a continuous linear map
\begin{center}
  $(P,P'):[X\oplus X' , Y\oplus Y']_s \longrightarrow[X,Y]_s \oplus [X',Y']_s$,  \\
  $\; \; \; \; \; \; (u)\mapsto (P(u),P'(u))$.
  \end{center}  
The same way applying Theorem~\ref{Interpolation} on the injections
\begin{center}
$I:Y\longrightarrow Y\oplus Y'$ and $I':Y'\longrightarrow Y\oplus Y'$,
\end{center}
we get a continuous linear map 
\begin{center}
$(I,I'):[X,Y]_s \oplus [X',Y']_s \longrightarrow [X\oplus X' , Y\oplus Y']_s$,  \\
$\; \; \; \; (z,z')\mapsto z\oplus z'$.
\end{center}

It is clear that $(I,I') \circ (P,P')=Id$ and $(P,P') \circ (I,I')=Id $.
\end{proof}
\end{lem}
  
\begin{lem}\label{Injective} The functor $F_{s}:\mathcal{C}\longrightarrow \mathcal{A}$ is injective with respect to the exact structure $(\mathcal{C} , \mathcal{E} _0)$.
\begin{proof}
By Proposition~\ref{Inj-functors}, it is enough to prove that $F_{s}$ preserves the kernel of every $\mathcal{E}_0$-strict morphism. Take $f:(X,Y)\longrightarrow (X',Y')$ a $\mathcal{E}_0$-strict morphism. Then there exist an admissible epimorphism $e:(X,Y)\longrightarrow (Z,W)$ and an admissible monomorphism $m:(Z,W)\longrightarrow (X',Y')$ such that we have a decomposition
\begin{center}
\begin{tikzcd}
(X,Y) \arrow[rr, "f"] \arrow[rd, "e"'] & {} \arrow[loop, distance=2em, in=305, out=235] & (X',Y') \\
                                   & (Z,W) \arrow[ru, "m"']                             &  
\end{tikzcd}
\end{center}
 By Remark 3.28 in \cite{E}, if $k_f:K_f\longrightarrow (X,Y)$ is the kernel of $f$, then $(k_f,e)\in \mathcal{E}_0$. An easy computation shows that the kernel of $f$ is the  morphism  

\begin{center}
$k_f:K_f=(X\cap Ker(f), Ker(f))\longrightarrow (X,Y)$.\\
$\; \; \; u \longrightarrow k_f(u)=u$.
\end{center}
Here $Ker(f)$ is given the norm of $Y$, and $X\cap Ker(f)$ is given the norm
$$\| u \| _{X\cap Ker(f)}=\max\{\| u \| _{X}, \| u \| _{Ker(f)} \}.$$
By Lemma 3.8 in \cite{E}, there exists a morphism $P:(X,Y)\longrightarrow K_f$ such that $P\circ k_f = Id_{K_f}$, and this means that $(X\cap Ker(f), Ker(f))$ is a complemented sub-couple of $(X,Y)$. Hence, by \cite[Section1.17.1, Theorem 1]{IN}, we have 
\begin{center}
 $[X\cap Ker(f), Ker(f)]_s=Ker(f)\cap [X,Y]_s=Ker(F_{s}(f))$.
 \end{center} 
\end{proof}

\end{lem}

Now we have the KC-pair in the category $\mathcal{C}$
\begin{center}
$ ( W^{1,2} (U\cup V), L^2(U\cup V) )\to (W^{1,2} (U)\oplus W^{1,2} (V),  L^2 (U)\oplus L^2 (V)) \to (W^{1,2} (U\cap V), L^2 (U\cap V))  $.
\end{center}

And by the assumption of the existence of $(W^{1,2},W^{0,2})$-double extension, this sequence splits, so it is in the structure $\mathcal{E}_0$. Hence, by Lemma~\ref{Injective}, if we apply the functor $F_{s}$ (for any $s\in ]0,1[$) we get an exact sequence. Therefore, by $(3.4)$ we get the  exact sequence     
 $$0\to  W^{s,2} (U\cup V)\to [W^{1,2} (U)\oplus W^{1,2} (V),  L^2 (U)\oplus L^2 (V)]_{s} \to [W^{1,2} (U\cap V), L^2 (U\cap V)]_{s}.$$
By Lemma~\ref{some} and $(3.4)$ we can write it the following way
\begin{center}
$0\to  W^{s,2} (U\cup V)\to W^{s,2} (U)\oplus W^{s,2} (V) \to [W^{1,2} (U\cap V), L^2 (U\cap V)]_{s}  $.
\end{center}
Hence, we have the exactness of the  sequence
 \begin{center}
 $0\to  W^{s,2} (U\cup V)\to W^{s,2} (U)\oplus W^{s,2} (V) \to W^{s,2} (U\cap V) \to 0.  $
 \end{center}
 
 \begin{flushright}
 $\square$
 \end{flushright}

\begin{rmk}
The answer to the exactness of the sequence 
\begin{center}
 $0\to  W^{s,2} (U\cup V)\to W^{s,2} (U)\oplus W^{s,2} (V) \to W^{s,2} (U\cap V) \to 0  $
 \end{center}

is important. A positive affirmation of its exactness would imply the possibility of sheafifying Sobolev spaces. Conversely, a negative outcome would indicate that there exist no degree-independent extension operators from $W^{i,2}(\Omega)$ to $W^{i,2}(\mathbb{R}^n)$ (for $i\in \{[s], [s]+1 \}$) when $\Omega$ is a cuspidal domain.
\end{rmk}
\section{Further discussion}
\begin{itemize}
\item  It may be helpful to construct a broader exact structure $\mathcal{E}$ on the category $\mathcal{C}$ of GB-pairs, such that the KC-pair
\begin{center}
$(W^{1,2} (U\cup V), L^2(U\cup V)) \to (W^{1,2} (U)\oplus W^{1,2} (V), L^2 (U)\oplus L^2 (V)) \to (W^{1,2} (U\cap V), L^2 (U\cap V)) \dotsb (\star)$
\end{center}
\vspace{0.2cm}
is in $\mathcal{E}$ (when $U$, $V$, and $U\cup V$ are Lipschitz domains). For example, we can demonstrate that the maximal class of all KC-pairs is exact on $\mathcal{C}$ (although this is not always true, as seen in \cite{E}). However, a challenge arises when enlarging the class $\mathcal{E}$, as this also broadens the class of $\mathcal{E}$-strict morphisms. For instance, if we consider $\mathcal{E}$ as the maximal class, a morphism $f:(X,Y)\longrightarrow (X',Y')$ is $\mathcal{E}$-strict if and only if $f(Y)$ is closed in $Y'$, $f(X)$ is closed in $X'$, $f$ is open into $f(Y)$, and $f\mid_{X}$ is open into $X'$. Nonetheless, at present, no result establishes the compatibility of interpolation with the kernel of such morphisms. In \cite{JL} and \cite{SJ}, some sufficient conditions for morphisms to have a kernel that is compatible with interpolation are provided. However, connecting these conditions with our specific situation remains unclear. Hence, it might be possible to devise an exact structure for the category of GB-pairs $\mathcal{C}$ that encompasses the KC-pair $(\star)$ and simultaneously satisfies the conditions outlined in \cite{JL} and \cite{SJ} for the class of strict morphisms."  
\item  Sheafification of Sobolev spaces in the usual sense in higher dimensions is much more challenging and remains unclear to us. Therefore, this requires a sheafification in the derived sense, as was achieved for negative regularity by G. Lebeau \cite{L} (building upon the work of Guillermou-Schapira \cite{GS} and Parusiński \cite{Pa2}). The two-dimensional case can be summarized with the following idea: take $U$ and $V$ as two cuspidal domains in $X_{\mathcal{A}}(\mathbb{R}^2)$, such that $U \cup V$ and $U \cap V$ are also cuspidal (see Figure 15).
\begin{figure}[H]
  \centering
  \includegraphics[scale=0.4]{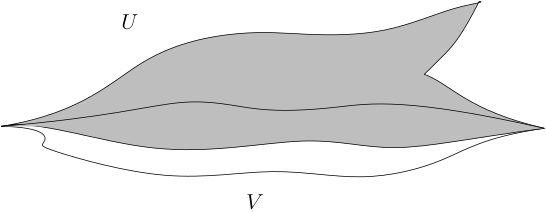}
  \label{casep2}
  \caption{$U$ and $V$.}   
\end{figure}
From the fact that we have enough space (from the metric point of view) outside $U$ and $V$, we can build two domains $\widehat{U} \in X_{\mathcal{A}}(\mathbb{R}^2)$ and $\widehat{V} \in X_{\mathcal{A}}(\mathbb{R}^2)$ with Lipschitz boundaries (see Figure 16) outside $U$ and $V$, such that $\widehat{U} \cup \widehat{V}$ has a Lipschitz boundary.
\begin{figure}[H]
  \centering
  \includegraphics[scale=0.4]{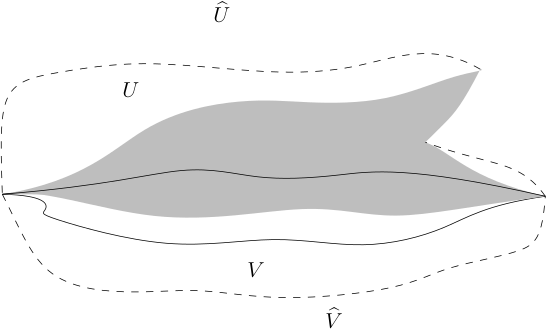}
  \label{casep2}
  \caption{$\widehat{U}$ and $\widehat{V}$.}   
\end{figure}
This gives a commutative diagram
\begin{center}
\begin{tikzcd}
            & 0                                 & 0                                 & 0                                 &   \\
0 \arrow[r] & W^{k,2}(U \cup V) \arrow[r] \arrow[u]             &  W^{k,2}(U) \oplus  W^{k,2}(V) \arrow[r] \arrow[u]             &  W^{k,2}(U \cap V) \arrow[r] \arrow[u]             & 0 \\
0 \arrow[r] & W^{k,2}(\widehat{U} \cup \widehat{V}) \arrow[r] \arrow[u] & W^{k,2}(\widehat{U}) \oplus  W^{k,2}(\widehat{V})  \arrow[r] \arrow[u] & W^{k,2}(\widehat{U} \cap \widehat{V}) \arrow[r] \arrow[u] & 0 \\
            & 0 \arrow[u]                       & 0 \arrow[u]                       & 0 \arrow[u]                       &  
\end{tikzcd}
\end{center}
  with the second exact line and exact rows. This implies the exactness of the first line.\\
  However, this flexibility is no longer true in higher dimensions, let's mention the following example (due to Parusi\'nski): take $U\in  X_{\mathcal{A}}(\mathbb{R}^3)$ and $U\in  X_{\mathcal{A}}(\mathbb{R}^3)$, both L-regular such that $U \cup V$ and $U \cap V$ are also L-regular (see the figure below). 
\begin{figure}[H]
  \centering
  \includegraphics[scale=0.4]{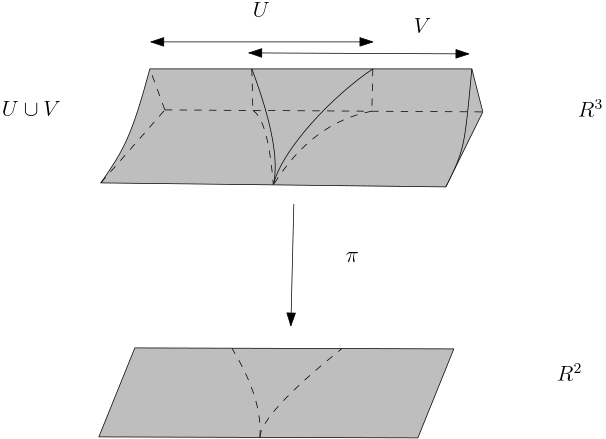}
\end{figure}
Then you can see directly that there is not enough space outside to build domains with Lipschitz boundaries and use Lemma~\ref{Proposition2.2}. This point is not clear and it is interesting to ask the following question:
\begin{question}
 For $k\in \mathbb{N}$, we define the presheaf
 \begin{center}
     $\mathcal{F}^k :X_{\mathcal{A}}(\mathbb{R}^n) \to \mathbb{C}$-vector spaces
 \end{center}
 such that for $U\in X_{\mathcal{A}}(\mathbb{R}^n)$, we have
 \begin{center}
     $\mathcal{F}^k (U)=\{ f\in L^2 (U) \; : \; f_{\mid K} \in W^{k,2} (K) \; for \; any  \; open \; L-regular \; K\subset U \}$.
 \end{center}
 Is $\mathcal{F}^k$ a sheaf on the site $X_{\mathcal{A}}(\mathbb{R}^n)$?
\end{question}
\end{itemize}

\end{document}